\documentclass[11pt,a4paper]{amsart}
\usepackage[latin2]{inputenc}
\usepackage{amsmath}
\usepackage{amsfonts}
\usepackage{pgfplots}
\usepackage{enumitem}
\usepackage{amssymb}
\usepackage{pgfplots}
\usepackage{array,booktabs}
\newcolumntype{M}[1]{>{\centering\arraybackslash}m{#1}}
\usepackage{float}
\usepackage{enumitem}
\usepackage{collcell}

\pgfplotsset{compat=newest}
\usepackage{graphicx}
\pgfplotsset{width=7cm,compat=1.16} 
\usepackage{pgf,tikz,pgfplots}
\pgfplotsset{compat=1.15}
\usepackage{mathrsfs}
\usetikzlibrary{arrows}

\usepackage{amsthm}
\usepackage[T1]{fontenc}
\usepackage{calc}
\usepackage{pstricks}
\usepackage{mathtools}

\usepackage[left=3.90cm, right=3.90cm, top=3.00cm, bottom=3.00cm]{geometry}
\usepackage{amsthm}
\usepackage[T1]{fontenc}
\usepackage{calc}
\usepackage{pstricks}
\usepackage{mathtools}
\usepackage{graphicx}
\usepackage{subcaption}
\usepackage{mciteplus}
\theoremstyle{definition}

\usepackage[colorlinks]{hyperref}
\hypersetup{colorlinks,linkcolor={blue},citecolor={blue},urlcolor={red}} 
\theoremstyle{plain}
\newtheorem{proposition}{Proposition}
\newtheorem{lemma}{Lemma}
\theoremstyle{plain}
\newtheorem{question}{Question}
\newtheorem{theorem}{Theorem}
\theoremstyle{remark}
\newtheorem*{remark2}{Remark}
\theoremstyle{remark}
\newtheorem*{remark3}{Acknowledgements}
\newcommand{\capac}{\operatorname{cap}}
\DeclareMathOperator*{\esssup}{ess\,sup}
\newcommand{\imag}{\operatorname{Im}}
\newcommand{\real}{\operatorname{Re}}

\title{Julia Sets of  Zorich maps}
\begin{document}
	\author{ATHANASIOS TSANTARIS}
	
	\maketitle
		\begin{abstract}
		The Julia set of the exponential family $E_{\kappa}:z\mapsto\kappa e^z$, $\kappa>0$ was shown to be the entire complex plane when $\kappa>1/e$ essentially by Misiurewicz. Later, Devaney and Krych showed that for $0<\kappa\leq1/e$ the Julia set is an uncountable union of pairwise disjoint simple curves tending to infinity.  Bergweiler generalized the result of Devaney and Krych for a three dimensional analogue of the exponential map called the Zorich map. We show that the Julia set of certain Zorich maps with  symmetry is the entire $\mathbb{R}^3$ generalizing Misiurewicz's result. Moreover, we show that the periodic points of the Zorich map are dense in $\mathbb{R}^3$ and that its escaping set is connected, generalizing a result of Rempe. We also generalize a theorem of Ghys, Sullivan and Goldberg on the measurable dynamics of the exponential.
	\end{abstract}
	\section{Introduction}
	\let\thefootnote\relax\footnote{2020 \textit{Mathematics Subject Classification}. Primary 37F10; Secondary 30C65, 30D05}

		In the study of the dynamics of complex analytic functions one of the most well studied and important families of functions is the exponential family $E_{\kappa}:z\mapsto\kappa e^z,$ $\kappa\in\mathbb{C}-\{0\}$. Perhaps the most fundamental fact about this family concerns its \textit{Julia set}. The Julia set  $\mathcal{J}(f)$ of an entire function $f$ is the set of all points in the complex plane where the family of iterates $f^n$ of $f$ is not normal. For $0<\kappa\leq1/e$, as was proven first by Devaney and Krych in \cite{Devaney1984}, the Julia set ${\mathcal{J}}(E_{\kappa})$ is a so called '\textit{Cantor bouquet}' which consists of uncountably many disjoint curves each of which has a finite endpoint and goes off to infinity. On the other hand, when $\kappa>1/e$  Misiurewicz in \cite{Misiu} proved that the Julia set ${\mathcal{J}}(E_{\lambda})$ equals the entire complex plane $\mathbb{C}$ (actually Misiurewicz only proved this for $\kappa=1$ but his proof can easily be adapted to cover the other cases as well, see \cite{Devaney2003}). For a different proof of the same fact see \cite{Shen2015}.  For all these facts and much more we refer to Devaney's survey \cite{Devaney2010} on exponential dynamics.
		
		In recent years there has been an increasing interest in the study of dynamics of quasiregular functions. Quasiregular functions are a higher dimensional generalization of holomorphic maps on the plane. We refer to \cite{Berg1} for a survey on the dynamics of such functions. As Bergweiler and Nicks have shown, in \cite{berg2013,B-Nicks}, there is a sensible definition for the Julia set for such functions which has many of the properties of the classical Julia set.  
		
		Moreover, in this higher dimensional setting, there is a whole family of maps that can be considered analogues of the exponential map called the \textit{Zorich maps} which are quasiregular and were first constructed by Zorich in \cite{Zorich}. Following \cite{Iwaniec2001} we describe the construction of the Zorich maps in three dimensions. Note that the construction can be done in arbitrary dimensions but we will confine ourselves in three dimensions for simplicity. First consider an $L$ bi-Lipschitz, sense-preserving map $\mathfrak{h}$ that maps the square \[Q:=\Big \{(x_1,x_2)\in\mathbb{R}^2:|x_1|\leq 1,|x_2|\leq 1\Big \}\] to the upper hemisphere \[\{(x_1,x_2,x_3)\in\mathbb{R}^3:x_1^2+x_2^2+x_3^2=1, x_3\geq 0\}.\] Then define $Z:Q\times\mathbb{R}\to\mathbb{R}^3$ as \[Z(x_1,x_2,x_3)=e^{x_3}\mathfrak{h}(x_1,x_2).\] The map $Z$ maps the square beam $Q\times \mathbb{R}$ to the upper half-space. By repeatedly reflecting now, across the sides of the square beam in the domain and the $x_1x_2$ plane in the range, we get a map $Z:\mathbb{R}^3\to\mathbb{R}^3$. Note that this map is doubly periodic meaning that $Z(x_1+4,x_2,x_3)=Z(x_1,x_2+4,x_3)=Z(x_1,x_2,x_3)$. Moreover, this map is not locally injective everywhere. The lines $x_1=2n+1, x_2=2m+1$, $n,m\in\mathbb{Z}$ belong to the \textit{branch set}, namely the set \[\mathcal{B}_{Z}:=\{x\in\mathbb{R}^3:Z \hspace{2mm}\text{is not locally homeomorphic at}\hspace{2mm}x\}.\] Also it can be shown that this map is quasiregular and has an essential singularity at infinity, just like the exponential map on the plane. Although we will not need this we call such quasiregular maps of \textit{transcendental type}. 
		
		We can also introduce a parameter $\nu>0$ and consider the family $Z_{\nu}=\nu Z$, where $Z$ is a Zorich map. This family can be considered as an analogue of the exponential family in higher dimensions (at least in the case where $\kappa>0$). Hence, it would be very interesting to know whether or not this family has a similar behaviour with the exponential in terms of dynamics. Indeed, Bergweiler in \cite{bergk} and Bergweiler and Nicks in \cite[Section 7]{B-Nicks} have proven that for small values of $\nu$ this family has as its Julia set uncountably many, pairwise disjoint curves. For those curves, Bergweiler in \cite{bergk} proved a counterpart to Karpinska's paradox (see \cite{Karpinska,Karpinska1999}) for the exponential map, namely the fact that the endpoints of those curves have Hausdorff dimension 3 while the curves minus the endpoints have Hausdorff dimension ~1. Moreover, Comd\"uhr in \cite{COMDUeHR2017} proved that those curves are smooth generalizing a result of Viana \cite{Silva1988} and in \cite{tsantaris2021} we studied the topology of those curves. It is also worth mentioning here that Cantor bouquets have been proven to exist for other generalized exponential functions, not necessarily quasiregular (see \cite{comduhr2019}).    Having said all that it seems quite reasonable to expect that for large values of $\nu$ the Julia set of the Zorich family would be the entire $\mathbb{R}^3$ just like in the exponential family where the Julia set is the entire complex plane. One aim of this paper is to prove that if we make some reasonable modifications to the map $\mathfrak{h}$ then this conjecture holds. 
		
		Let us now define the modified $\mathfrak{h}$ and state our main theorem. The first thing that we require is that our map $\mathfrak{h}(x_1,x_2)=(\mathfrak{h}_1(x_1,x_2),\mathfrak{h}_2(x_1,x_2),\mathfrak{h}_3(x_1,x_2))$ must satisfy $\mathfrak{h}_1(x_1,x_1)=\mathfrak{h}_2(x_1,x_1)$ and $\mathfrak{h}_1(x_1,-x_1)=-\mathfrak{h}_2(x_1,-x_1)$. This way the planes $x_1=x_2$ and $x_1=-x_2$ are invariant under the Zorich map we get. Note that this implies that $\mathfrak{h}(0,0)=(0,0,1)$. Second, we need to  scale things by a factor $\lambda>1$. To be more precise we define the function \[h(x_1,x_2)=\lambda \mathfrak{h}\left(\frac{1}{\lambda}(x_1,x_2)\right), \hspace{1mm} (x_1,x_2)\in\lambda Q.\]
		We define now the Zorich maps we get by this $h$, which we denote by $\mathcal{Z}$
		
		\begin{equation}\label{zorichdef}\mathcal{Z}_{\nu}(x_1,x_2,x_3)=\nu e^{x_3}h(x_1,x_2),\hspace{1mm} (x_1,x_2,x_3)\in \lambda Q\times \mathbb{R},\hspace{1mm}\nu>0.\end{equation}
		Again we extend this map to $\mathbb{R}^3$ by reflecting across the sides of the square beam and the plane $x_3=0$. Another important thing to note here is that during the extension process of our map $\mathcal{Z}_\nu$ from the initial square beam to the whole $\mathbb{R}^3$ we can also extend $\mathfrak{h}$ to a Lipschitz map from $\mathbb{R}^2$ to $\mathbb{R}^3$ with the same Lipschitz constant $L$. We will always assume that this extension has been done and when we talk about $\mathfrak{h}$ we will mean the extended one unless otherwise stated. Moreover, let us note here that this new Zorich map $\mathcal{Z}$ is conjugate to $x\mapsto Z(x_1,x_2,\lambda x_3)$, where $Z$ is the classic Zorich map without the scaling.
		\begin{remark2}
			Here it is worth elaborating on that last sentence. Instead of studying the family $\mathcal{Z}_\nu$, defined in \eqref{zorichdef}, we could have studied the family $\alpha \circ Z$, where $\alpha:\mathbb{R}^3\to\mathbb{R}^3$ is the linear map induced by the matrix $$\begin{pmatrix}
				\nu& 0&0\\
				0& \nu &0\\
				0&0&\nu\lambda
			\end{pmatrix} $$ and $Z$ is the Zorich map that leaves the planes $x_1=\pm x_2$ invariant and comes from using $\mathfrak{h}$. It is easy to see that the map  $\alpha \circ Z$ is conjugate with $\nu Z(x_1,x_2,\lambda x_3)$ and thus with $\mathcal{Z}_\nu$. The advantage of this viewpoint is that the Zorich maps we consider here and the maps considered by Bergweiler in \cite{bergk} can all be seen as maps in the space $\{\mathcal{A}\circ Z:\mathcal{A}\in GL_3(\mathbb{R})\}$, where $GL_3(\mathbb{R})$ is the general linear group of degree $3$. Thus $GL_3(\mathbb{R})\setminus\{0\}$ becomes the parameter space for Zorich maps in analogy with $\mathbb{C}\setminus\{0\}$ being the parameter space for the exponential map.
			
			Due to the conjugacy all of the theorems we are going to prove here are also true for the family $\alpha\circ Z$. We have chosen to use a different  presentation of Zorich maps than the one described here since that way the definition seems more natural and  the connection with the exponential family  is more apparent.

		\end{remark2}
		For the type of Zorich maps defined in \eqref{zorichdef} we will prove
		\begin{theorem}\label{main}
			Let $\lambda>L^{5}$. Then for all $\nu> \sqrt{\frac{2L}{\lambda}}$  the Zorich map $\mathcal{Z}_{\nu}$ we get using this scale factor $\lambda$ has as its Julia set  the whole $\mathbb{R}^3$.
		\end{theorem} 
		\begin{remark2}
			We will actually prove a slightly stronger result. Namely, that if the assumptions of the above theorem are satisfied and $V$ is any open set of $\mathbb{R}^3$ then $\bigcup_{n\geq0}\mathcal{Z}_{\nu}^n(V)$ covers $\mathbb{R}^3\setminus\{0\}$. 
		\end{remark2}
		
		The above Theorem implies that the behaviour of the iterates of the Zorich maps, for those particular choices of the parameters, are chaotic in the whole $\mathbb{R}^3$. Along the way of proving Theorem \ref{main} we will also prove a Theorem on the measurable dynamics of Zorich maps which can be seen as the analogous result of a theorem for exponential maps due to Ghys, Sullivan and Goldberg (see Theorem \ref{fibers} in section \ref{remarks} for more details). Another fact usually associated with chaotic behaviour in a set is the density of periodic points on that set. In the complex plane it is well known and was first proven by Baker in \cite{Baker1968}, that periodic points of an entire transcendental map (in fact even repelling periodic points) are dense in its Julia set. However, it still unknown whether or not the periodic points of a quasiregular map on $\mathbb{R}^3$ are dense in its Julia set. We are able to prove that this is indeed the case for Zorich maps.
		
		\begin{theorem}\label{density}
			Let $\nu$ and $\lambda$ be as in Theorem \ref{main}. The periodic points of $\mathcal{Z_{\nu}}$ are dense in $\mathbb{R}^3$.	
		\end{theorem}
		
		Another object of study in the exponential family, and in transcendental complex dynamics in general, which is intimately connected with  the Julia set is the escaping set. It was first studied by Eremenko in \cite{Eremenko} and if $f:\mathbb{C}\to\mathbb{C}$ is an entire function then it is defined as \[I(f):=\{z\in \mathbb{C}:|f^n(z)|\to\infty \hspace{2mm}\text{as}\hspace{2mm}n\to\infty\}.\]Eremenko proved that $I(f)\not=\emptyset$ and that $\partial I(f)=\mathcal{J}(f)$.
		
		Moreover, for the exponential family from \cite{Eremenko1992} it is true that ${I(f)}\subset \mathcal{J}(f)$ and thus $I(f)$ is dense in the Julia set. When the Julia set is a Cantor bouquet, the escaping set consists of the disjoint curves that make up the Julia set together with some of their endpoints. In this case $I(f)$ is disconnected while $I(f)\cup\{\infty\}$ is connected (see \cite{Devaney2010}). On the other hand, when the Julia set of a map in the exponential family is the entire complex plane, the escaping set is dense in the complex plane and Rempe in \cite{REMPE2010} has proven that it is also connected.\\
		
		The situation is similar for the Zorich maps as well. As we already mentioned, in \cite{bergk,B-Nicks} it is proven that for some values of the parameter $\nu$ the Julia set consists of disjoint curves together with their endpoints and $I(\mathcal{Z}_{\nu})$ is again a disconnected subset of the Julia set. On the other hand we are able to show that 
		\begin{theorem}\label{escaping}
			For the same choice of $\nu$ and $\lambda$ as in Theorem \ref{main}, we have that the escaping set $I(\mathcal{Z}_{\nu})$ is a connected subset of $\mathbb{R}^3$.
		\end{theorem}

		It is also worth mentioning here that there are other methods of constructing Zorich-like maps where instead of mapping squares to hemispheres through bi-Lipschitz functions we map squares to surfaces whose boundary lies on the plane $x_3=0$ and the half ray connecting the origin with a point on the surface intersects the surface only once. If we further impose some bound on the angle between that ray and the tangent plane to the surface (see section \ref{pyramida} for more details) we can use our methods and prove a Theorem similar to Theorem \ref{main}.
		
		To state the theorem let us denote those Zorich maps with $\mathcal{Z}_{gen}$. In the construction of those maps we will use a bi-Lipschitz map $h_{gen}$ which will be the rescaled version, by a factor of $\lambda$, of another $L$ bi-Lipschitz map $\mathfrak{h}$.
		Note that we do not introduce the parameter $\nu$ in this case for simplicity.

		\begin{theorem}\label{pyramid}
			For $\lambda> C_{h_{gen}}$ the Julia set $\mathcal{J}(\mathcal{Z}_{gen})$  is the entire $\mathbb{R}^3$, where  $C_{h_{gen}}$ a constant depending on $h_{gen}$.
		\end{theorem}
		
		The proof of this theorem is essentially the same as the one we gave for Theorem~ \ref{main}. We will give a sketch of the proof in  section \ref{pyramida} where we also find an explicit value for the constant $C_{h_{gen}}$.  
		
		Theorems \ref{density} and \ref{escaping} possibly also hold for those kind of Zorich maps with very similar proofs although we forgo the effort of proving them here.  \\
		
		The structure of the rest of paper is as follows. In section 2 we give some definitions and some preliminary results. In section 3 we study our Zorich maps on the planes $x_1=\pm x_2$. In section 4 we prove Theorem \ref{main}. Section 5 is dedicated to the escaping set and the proof of Theorem \ref{escaping}. In section 6 we prove Theorem \ref{density} while  in section 7 we discuss the more general Zorich maps. Finally, in section 8 we discuss some further questions and prove a theorem on the measurable dynamics of the Zorich maps.
		\begin{remark3}
			I would like to thank my supervisor  Daniel Nicks for all his help and encouragement while writing this paper. I would also like to thank Alastair Fletcher, Daniel Meyer and the anonymous referee for their questions which led to a more general version of Theorem \ref{pyramid} and the referee for their thoughtful comments which improved the exposition.
		\end{remark3}

		\section{Preliminaries and background on quasiregular dynamics}
		Here we give a brief overview of the terminology and the notation we are going to need. For definitions and a more detailed treatment of quasiregular maps we refer to \cite{Rickman, vuorinen}. For a survey in the iteration of such maps we refer to \cite{Berg1}.\\\\
		Here we will just note that if  $d\geq 2$ and $G\subset \mathbb{R}^d$ is a domain and $f=(f_1,f_2,\cdots, f_d):G\to\mathbb{R}^d$ is a differentiable map we will denote the total derivative of this map by
		$Df(x)$. Also $|Df(x)|$ denotes the operator norm of the derivative, meaning $$|Df(x)|=\sup_{|h|=1}|Df(x)(h)|,$$ and $J_f(x)$ denotes the Jacobian determinant. Moreover we set $$\ell(Df(x))=\inf_{|h|=1}|Df(x)(h)|.$$

		We will also need the notion of the capacity of a condenser in order to define the Julia set of a quasiregular map. A condenser in $\mathbb{R}^d$ is a pair $E=(A,C)$, where $A$ is an open set in $\mathbb{R}^d$ and $C$ is a compact subset of $A$. The \textit{conformal capacity} or just \textit{capacity} of the condenser $E$ is defined as 
		$$\capac E=\inf_{u}\int_{A}|\nabla u|^ddm,$$
		where the infimum is taken over all non-negative functions $u\in C_0^{\infty}(A)$ which satisfy $u_{|C}\geq 1$ and $m$ is the $d$-dimensional Lebesgue measure.\\\\
		If $\capac (A,C)=0$ for some bounded open set $A$ containing $C$, then it is also true that $\capac (A',C)=0$ for every other bounded set $A'$ containing $C$;\cite[Lemma III.2.2]{Rickman}. In this case we say that $C$ has zero capacity and we write $\capac C=0$; otherwise we say that $C$ has positive capacity and we write $\capac C>0$. Also for an arbitrary set $C\subset \mathbb{R}^d$, we write $\capac C=0$ when $\capac F=0$ for every compact subset $F$ of $C$. If the capacity of a set is zero then this set has Hausdorff dimension zero \cite[Theorem~VII.1.15]{Rickman}. Thus a zero capacity set is small in this sense. It is also quite easy to see that for any two sets $S,B$ with $S\subset B$ if $\capac B=0$ then $\capac S=0$.\\

		In \cite{berg2013} Bergweiler developed a Fatou-Julia theory for  quasiregular self-maps of $\overline{\mathbb{R}^d}$, which include polynomial type quasiregular maps, and can be thought of as analogues of rational maps, while in \cite{B-Nicks} Bergweiler and Nicks did the same but for transcendental type quasiregular maps. Following those two papers we define the Julia set of $f:\mathbb{R}^d\to \mathbb{R}^d$, denoted ${\mathcal{J}}(f)$, to be the set of all those $x\in\mathbb{R}^d$ such that
		$$ \capac \left(\mathbb{R}^d\setminus \bigcup_{k=1}^\infty f^k(U)\right)=0
		$$	for every neighbourhood $U$ of $x$. We call the complement of ${\mathcal{J}}(f)$ the \textit{quasi-Fatou set}, and we denote it by $QF(f)$.  
		
		Note here that we used something like the blow-up property, that the Julia set  in complex dynamics has (see for example \cite[Theorem 4.10]{milnor}), in order to define our Julia set. Also note that we do not assume anything about the normality of the family of iterates of $f$ in the quasi-Fatou set. It turns out that the Julia set we defined enjoys many of the properties that the classical Julia set for holomorphic maps has. In particular, a property that we will use in this paper is that the Julia set is a \textit{completely invariant} set, meaning that $x\in\mathcal{J}(f)$ if and only if $f(x)\in \mathcal{J}(f)$. For more details and motivation behind the definition of the Julia set we refer to \cite{Berg1,berg2013,B-Nicks}. 
		
		As we already have noted, the Zorich map is a quasiregular map of transcendental type and thus the above definitions make sense for this map. 
		
		In this section we will also prove 
		\begin{proposition}\label{prop}
			The $x_3$-axis belongs in $\mathcal{J}(\mathcal{Z_{\nu}})$ for all $\lambda\geq1$ and $\nu$ with $\lambda\nu> 1/e$.	
		\end{proposition}
		Before we prove this proposition let us name a few things first. Using the same notation as in \cite{bergk}, for $r=(r_1,r_2)\in\mathbb{Z}^2$ we define \[P(r)=P(r_1,r_2):=\{(x_1,x_2)\in\mathbb{R}^2:|x_1-2\lambda r_1|<\lambda, |x_2-2\lambda r_2|<\lambda\}.\] For $c\in\mathbb{R}$ we also define $H_{>c}$ to be the half-space $\{(x_1,x_2,x_3)\in\mathbb{R}^3:x_3>c\}$ and we define $H_{\geq c}$, $H_{<c}$ similarly. We observe here that $\mathcal{Z}_{\nu}$ maps $P(r_1,r_2)\times \mathbb{R}$ bijectively to $H_{>0}$, when $r_1+r_2$ is even and to $H_{<0}$ when $r_1+r_2$ is odd.  Thus there is an inverse branch $\Lambda_{(0,0)}:H_{>0}\to P(0,0)\times\mathbb{R}$. We can now, as in \cite{bergk}, find constants $M_0\in\mathbb{R}$ and $\alpha\in (0,1)$ such that \begin{equation}\label{eq101}|\Lambda_{(0,0)}(x)-\Lambda_{(0,0)}(y)|\leq \alpha|x-y|,\hspace{1mm} \text{for all}\hspace{1mm}x,y\in H_{>\nu\lambda e^{M_0}}.\end{equation}
		The next lemma is similar to \cite[Lemma 7.1]{B-Nicks}.
		\begin{lemma}\label{lemma0}
			Let $M>M_0>0$ be a large positive number and $x\in \Lambda_{(0,0)}(H_{>\nu\lambda e^M})$. Then \begin{equation}\label{eq1}\mathcal{Z}_{\nu}(B(x,R)\cap H_{\geq M})\supset B(\mathcal{Z}_{\nu}(x),\alpha^{-1}R )\cap H_{>\nu\lambda e^M},\end{equation} where $R>0$ and $B(x,R)$ denotes the ball of centre $x$ and radius $R$.
		\end{lemma}
		\begin{proof}
			Note that $x\in \Lambda_{(0,0)}(H_{>\nu\lambda e^M})$ implies that $x\in P(0,0)\times (M,\infty)$.
			Let $$y\in B(\mathcal{Z}_{\nu}(x),\alpha^{-1}R )\cap H_{>\nu\lambda e^M}.$$ Then by (\ref{eq101}) we have that \[|x-\Lambda_{(0,0)}(y)|=|\Lambda_{(0,0)}(\mathcal{Z}_{\nu}(x))-\Lambda_{(0,0)}(y)|\leq \alpha|\mathcal{Z}_{\nu}(x)-y|<R.\]
			Hence, $\Lambda_{(0,0)}(y)\in B(x,R)\cap P(0,0)$ and thus $y=\mathcal{Z}_{\nu}(\Lambda_{(0,0)}(y))\in \mathcal{Z}_{\nu}(B(x,R)\cap H_{\geq M})$.
		\end{proof}
		\begin{proof}[Proof of Proposition \ref{prop}]

			Let us fix a point $x=(0,0,x_0)$ on the $x_3$-axis and consider a neighbourhood $U$ of that point. It is easy to see now that $\mathcal{Z}^{k}_{\nu}(x)=(0,0,E^k_{\nu\lambda}(x_0))$, where $E_{\nu\lambda}^k$ denotes the $k$-th iterate of the map $E_{\nu\lambda}(t)=\nu\lambda e^t$. Since the $x_3$-axis is invariant under our Zorich map and since  $\nu\lambda> 1/e$ we have that $E_{\nu\lambda}^k(x)\to \infty$ and thus we may assume that $x\in H_{\geq M}$, for some $M>M_0$. By repeatedly applying (\ref{eq1}) we may now obtain a sequence $R_k\to \infty$ with \[\mathcal{Z}_{\nu}^k(U)\supset B(\mathcal{Z}_{\nu}^k(x),R_k)\cap H_{E_{\nu\lambda}^k(M)}\] and we note that the intersection on the right hand side  always contains the upper half of the ball $B(\mathcal{Z}_{\nu}^k(x),R_k)$. Hence, for large enough $k$, the set \[V_k=\{(x_1,x_2)\in\mathbb{R}^2:|x_1|\leq 2\lambda,|x_2|\leq2\lambda\}\times \left[E_{\nu\lambda}^k(x_0),E_{\nu\lambda}^k(x_0)+R_k/2\right],\]  is contained in $B(\mathcal{Z}_{\nu}^k(x),R_k)\cap H_{E_{\nu\lambda}^k(M)}$. Observe now that $\mathcal{Z}_{\nu}$ maps $V_k$ onto the shell \[A_k=\{x\in\mathbb{R}^3:\nu\lambda \exp\left({E_{\nu\lambda}^k(x_0)}\right)\leq|x|\leq \nu\lambda \exp\left({E_{\nu\lambda}^k(x_0)+R_k/2}\right)\}.\] It is easy to see now that this shell, for large enough $k$ and since $R_k\to \infty$, contains a set of the form \[\{(x_1,x_2,x_3)\in\mathbb{R}^3:|x_1-2\lambda q_{k,1}|\leq 2\lambda,|x_2-2\lambda q_{k,2}|\leq2\lambda,|x_3|\leq t_k\},\] with $q_{k,1}, q_{k,2}\in\mathbb{Z}$ and $t_k\to \infty$. This implies that \[\{x\in\mathbb{R}^3:\nu\lambda e^{-t_k}\leq|x|\leq \nu\lambda e^{t_k}\}\subset \mathcal{Z}_{\nu}(A_k)\subset\mathcal{Z}_{\nu}^{k+2}(U).\] Hence $\bigcup_{k=1}^{\infty}\mathcal{Z}_{\nu}^k(U)=\mathbb{R}^3\setminus\{(0,0,0)\}$ which means that $x\in\mathcal{J}(\mathcal{Z}_{\nu})$.
		\end{proof}
		\section{The Zorich map on the planes $x_1=\pm x_2$}\label{section3}
		Here we will prove some basic facts about the Zorich family we have constructed. As we already have mentioned in the introduction our Zorich maps send the planes $x_1=x_2$ and $x_1=-x_2$ to themselves. We would like to know the behaviour of $\mathcal{Z}_{\nu}$ restricted to those planes. With that in mind, we observe that restricted to the plane $x_1=x_2$ our Zorich map is conjugate through  $\phi(x_1,x_1,x_3)=\frac{1}{\lambda}(x_3+i\sqrt2 x_1)$ to the map $g:\mathbb{C}\to\mathbb{C}$ given by \[g(z):= \begin{cases}
			\psi(	\bar{z}+2 \sqrt2i),&\hspace{1mm}\imag (z)\in\left[(4k+1)\sqrt2,(4k+3)\sqrt2\right]\\\\\psi(z),&\hspace{1mm}\imag (z)\in\left[(4k-1)\sqrt2,(4k+1)\sqrt2\right],
		\end{cases}\]where  $k\in\mathbb{Z}$ and $\psi(x+iy)= \nu e^{\lambda x}\left(\mathfrak{h}_3\left(\frac{y}{\sqrt2},\frac{y}{\sqrt2}\right)+i\sqrt2\mathfrak{h}_1\left(\frac{y}{\sqrt2},\frac{y}{\sqrt2}\right)\right)$. Similarly the Zorich map is conjugate to a similar map to $g$ on the plane $x_2=-x_1$ and everything that follows works in that case as well. For simplicity let us write $a(y)$ and $b(y)$ instead of $\mathfrak{h}_3(y/\sqrt2,y/\sqrt2)$ and $\sqrt2\mathfrak{h}_1(y/\sqrt2,y/\sqrt2)$. Note that $a^2(y)+b^2(y)=1$. Also let us note here that the function $\psi$ is quasiregular and that $g(\mathbb{C})=\{\real z>0 \}$. Furthermore, $g$ is not a quasiregular map, since it is not sense preserving, and is a two to one function in the strip $\{z\in\mathbb{C}:(4k-1)\sqrt2\leq\imag(z)\leq(4k+3)\sqrt2\}$.

		We would now like to show that the planes $x_1=\pm x_2$ belong to the Julia set of $\mathcal{Z}_{\nu}$. We already know, from Proposition \ref{prop}, that the $x_3$-axis belongs to the Julia set. With that in mind we will prove that any open set in $\mathbb{R}^3$ that intersects those planes also intersects the $x_3$-axis under iteration by $\mathcal{Z}_\nu$. Now since we know that $\mathcal{Z}_{\nu}$ is conjugate to $g$ on those planes it is enough to prove that any open set in the complex plane intersects the real axis under iteration by $g$.

		\begin{theorem}\label{plane}
			Let $\nu^2\lambda>2L$ and $V\subset\mathbb{C}$ be a connected set with $m(V)>0$, where $m$ is the 2 dimensional Lebesgue measure. Then $g^n(V)$ intersects the real axis for some $n\in\mathbb{N}$.
		\end{theorem}
		For the proof of this Theorem we will need several lemmas. Note here that since $\mathfrak{h}$ is a Lipschitz function it will also be differentiable almost everywhere. This implies that $g$ is differentiable almost everywhere.
		\begin{lemma}\label{expand}
			
			\[|\det (Dg(z))|\geq\frac{\nu^2\lambda e^{2\lambda\real(z)}}{L}\hspace{1mm}\text{a.e.}\]
		\end{lemma}
		\begin{proof}
			It is enough to find a lower bound for $\imag z\in\left[(4k-1)\sqrt2,(4k+1)\sqrt2\right]$. This is true because for other $z$ we have that $T(z)=\bar{z}+2\sqrt2 i$ has imaginary part in $\left[(4k'-1)\sqrt2,(4k'+1)\sqrt2\right]$ for some $k'\in\mathbb{Z}$ and thus for those $z$ we have $g(z)=g(T(z))=\psi(T(z))$.  Then by the chain rule we have that $Dg(z)=Dg(T(z))DT(z)$. Since $DT(z)=-1$ this implies that $\left|\det Dg(z)\right|=\left|\det Dg(T(z))\right|$.

			With than in mind, if $z=x+iy$ we have that $Dg(z)$ is the linear transformation induced by the matrix \[\begin{pmatrix}
				\nu\lambda e^{\lambda x}a(y)& \nu e^{\lambda x} \frac{da}{dy}({y})\\\\
				\nu\lambda e^{\lambda x}b(y)& \nu e^{\lambda x} \frac{db}{dy}(y)
			\end{pmatrix}=\begin{pmatrix}
				\nu\lambda e^{\lambda x}\mathfrak{h}_3(\frac{y}{\sqrt2})& \frac{\nu}{\sqrt2} e^{\lambda x} \frac{d\mathfrak{h}_3}{dy}(\frac{y}{\sqrt2})\\\\
				\nu\lambda\sqrt2 e^{\lambda x}\mathfrak{h}_1(\frac{y}{\sqrt2})& \nu e^{\lambda x} \frac{d\mathfrak{h}_1}{dy}(\frac{y}{\sqrt2})
			\end{pmatrix},\] $\hspace{1mm}\text{if}\hspace{1mm} y\in\left[(4k-1)\sqrt2,(4k+1)\sqrt2\right]$.
			
			Thus \begin{align*}|\det Dg(z)|&=\nu^2\lambda e^{2\lambda x}\left|\mathfrak{h}_3(y)\frac{d\mathfrak{h}_1}{dy}(y)-\mathfrak{h}_1(y)\frac{d\mathfrak{h}_3}{dy}(y)\right|\\&=\nu^2\lambda e^{2\lambda x}\left|\det\begin{pmatrix} a(y)& \frac{da}{dy}({y})\\\\
					b(y)& \frac{db}{dy}(y)\end{pmatrix}\right|, \hspace{1mm}\text{for}\hspace{1mm}y\in\left[(4k-1)\sqrt2,(4k+1)\sqrt2\right].\end{align*} 
			
			We now claim that  \[\left|\det\begin{pmatrix} a(y)& \frac{da}{dy}({y})\\\\
				b(y)& \frac{db}{dy}(y)\end{pmatrix}\right|>\frac{1}{L},\hspace{2mm}\text{a.e.}\]Hence \[|\det Dg(z)|\geq\frac{\nu^2\lambda e^{2\lambda x}}{L}\hspace{2mm}\text{a.e.}\] Indeed, because $a^2(y)+b^2(y)=1$ we get that the vectors $(a(y),b(y))$ and $(\frac{da}{dy},\frac{db}{dy})$ are orthogonal and thus so is their matrix. This implies that \[\left|\det\begin{pmatrix} a(y)& \frac{da}{dy}({y})\\\\
				b(y)& \frac{db}{dy}(y)\end{pmatrix}\right|=\left|\left(a(y),b(y)\right)\right|\left|\left(\frac{da}{dy}({y}),\frac{db}{dy}({y})\right)\right|=\left|\left(\frac{da}{dy}({y}),\frac{db}{dy}({y})\right)\right|.\]
			
			Now because $\mathfrak{h}$ is a locally bi-Lipschitz map almost everywhere we have that $$\left|\frac{d\mathfrak{h}}{dy}\left(\frac{y}{\sqrt2},\frac{y}{\sqrt2}\right)\right|\geq \frac{1}{L}\hspace{2mm}\text{a.e.}$$
			Since $\left|\left(\frac{da}{dy}({y}),\frac{db}{dy}({y})\right)\right|=\left|\frac{d\mathfrak{h}}{dy}\left(\frac{y}{\sqrt2},\frac{y}{\sqrt2}\right)\right|$ we get what we wanted.
			
		\end{proof}
		\begin{lemma}\label{volume}
			Let $\nu^2\lambda>2L$ where $\lambda\geq1$, $\nu>0$ and $V\subset \mathbb{C}$ be a connected subset of the complex plane with $m(V)>0$ and such that its iterates under $g$ do not intersect the real axis. Then $m(g^n(V))\to\infty$ as $n\to \infty$, where $m$ is the 2-dimensional Lebesgue measure.
		\end{lemma}
		\begin{proof}
			We can assume that $V$ lies on the right half plane $\{z:\real (z)>0\}$ otherwise just consider $g(V)$ since $g$ maps $\mathbb{C}$ to the right half plane. We know that $m(g(V))=\int_{g(V)}dm$. Since none of the iterates of $V$ under $g$ intersects the real axis we have that those iterates  also do not intersect any of the pre-images of the real axis, meaning the lines $y=2\sqrt2k,$ $k\in~\mathbb{Z}$. Thus $g^n(V)$ is always inside  strips of the form $\{z\in\mathbb{C}:\imag z\in\left(2\sqrt2k,2\sqrt2(k+1)\right)\}$. In each of those strips it is easy to see, by what we have said in the first paragraph of this section, that $g$ is a two-to-one map. By using Lemma \ref{expand} and since $\real z>0$ for all $z\in V$  we now get \[\int_{g(V)}dm\geq\frac{1}{2}\int_V |\det (Dg)|dm\geq\frac{\nu^2\lambda}{2L} \int_V e^{2\lambda\real (z)}dm\geq \frac{\nu^2\lambda}{2L}m(V).\] This means that \[m(g(V))\geq \frac{\nu^2\lambda }{2L}m(V).\]
			Hence, since $\nu^2\lambda>2L$ if we iterate that inequality we have that $m(g^n(V))\to\infty$.
			
		\end{proof}
		\begin{figure}[H]
			\centering
			\includegraphics[scale=0.9]{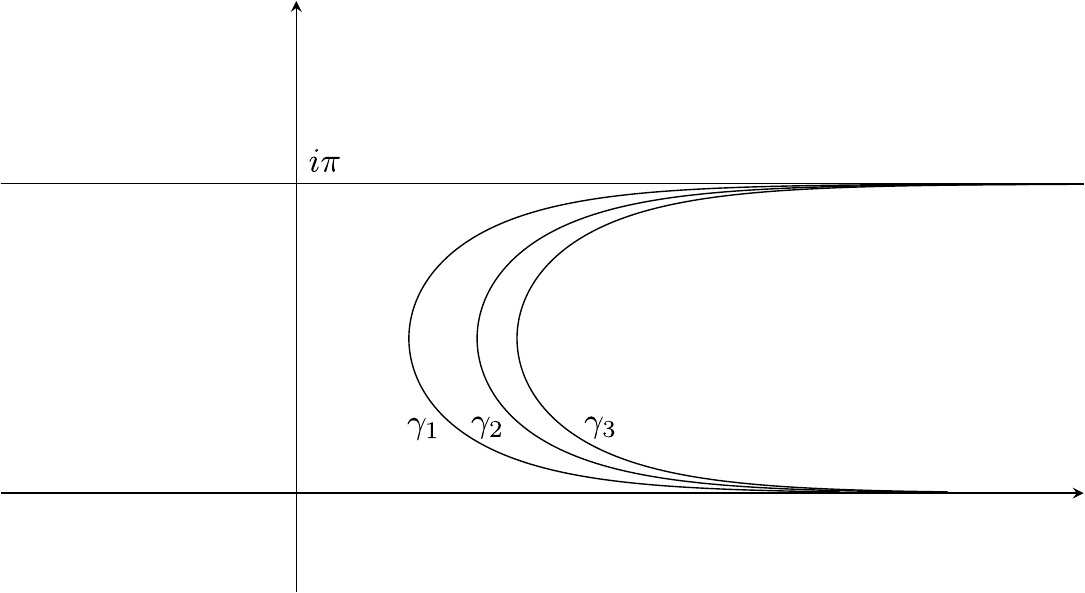}
			\caption{The pre-images of the lines $2k\pi i$, $k\in\mathbb{Z}\setminus\{0\}$ under the exponential map. The curves $\gamma_m$ in the proof of Theorem \ref{plane} have a similar structure. }
		\end{figure}
		\begin{proof}[Proof of Theorem \ref{plane}]
			
			Suppose, towards a contradiction, that there is a connected  set $V$ with $m(V)>0$ whose iterates never intersect the real axis. Then by Lemma \ref{volume} we have that $m(g^n(V))\to \infty$ as $n\to \infty$. Since $g^n(V)$ never intersects the real axis it also does not intersect its pre-images meaning the lines $\imag(z)=2\sqrt2k$, $k\in\mathbb{Z}$. This means that $g^n(V)$ stays always inside  strips of the form $\{ z\in\mathbb{C}:\imag(z)\in\left(2\sqrt2k,2\sqrt2(k+1)\right),k\in\mathbb{Z}\}$. If we now take the pre-image of the lines $\imag(z)=2\sqrt2m$, $m\in\mathbb{Z}\setminus\{0\}$, that lie inside all those strips we get curves $\gamma_m$ which the iterates $g^n(V)$ of our set must not cross. By symmetry we can confine ourselves in the strip $$S:=\Big\{ z\in\mathbb{C}:\imag(z)\in\left(0,2\sqrt2\right)\Big\}.$$ 
			
			From now on let us write $\mathfrak{h}_i(y)$ instead of $\mathfrak{h}_i(y,y)$, $i=1,2,3$ for simplicity. 
			We have two cases to consider now. Either $S$ contains the pre-images of the lines $\imag(z)=2\sqrt2m$, $m>0$, in which case $\mathfrak{h}_1(y)>0$ for $y>0$ or $S$ contains the pre-images of the lines for $m<0$ in which case $\mathfrak{h}_1(y)<0$ for $y>0$. We will only consider the first case here. The second one can be dealt similarly. So suppose $m>0$. We claim that the curves $\gamma_m$ partition the strip $S$ in sets of uniformly bounded area. In fact if we name $A_m$ the area of the set defined by $\gamma_m$ and $\gamma_{m+1}$  and we name $A_0$ the area inside the strip between the imaginary axis and $\gamma_1$, then we claim that $A_m$ is an eventually decreasing sequence. Clearly, since   $m(g^n(V))\to \infty$ and $g^n(V)$ must stay inside those sets $A_m$ this is impossible. 
			
			In order to prove those claims note that the curves $\gamma_m$ are given by the equations $\nu e^{\lambda x}\mathfrak{h}_1\left(\frac{y}{\sqrt2}\right)=2m,$ when $y\in\left(0,\sqrt2\right]$ and $\nu e^{\lambda x}\mathfrak{h}_1\left(\frac{2\sqrt2-y}{\sqrt2}\right)=2m$, when $y\in\left[\sqrt2,2\sqrt2\right)$. It is also easy to see that those curves do not have self intersections and do not intersect with each other.  The area we are looking for will be given by \[A_m=\int_{0}^{\sqrt2}\frac{1}{\lambda}\left(\log\frac{2(m+1)}{\nu \mathfrak{h}_1\left(\frac{y}{\sqrt2}\right)}-\log\frac{2m}{\nu \mathfrak{h}_1\left(\frac{y}{\sqrt2}\right)}\right)dy+\int_{\sqrt2}^{2\sqrt2}\frac{1}{\lambda}\left(\log\frac{2(m+1)}{\nu \mathfrak{h}_1\left(\frac{2\sqrt2-y}{\sqrt2}\right)}-\log\frac{2m}{\nu \mathfrak{h}_1\left(\frac{2\sqrt2-y}{\sqrt2}\right)}\right)dy.\] Thus $A_m=\frac{2\sqrt2}{\lambda}\log\frac{m+1}{m}$ which proves what we wanted. We also need to find $A_0$ for which it is true that  \[A_0=\int_{0}^{\sqrt2}\frac{1}{\lambda}\log \frac{2}{\nu \mathfrak{h}_1\left(\frac{y}{\sqrt2}\right)}dy+\int_{\sqrt2}^{2\sqrt2}\frac{1}{\lambda}\log\frac{2}{\nu \mathfrak{h}_1\left(\frac{2\sqrt2-y}{\sqrt2}\right)}dy=2\int_{0}^{\sqrt2}\frac{1}{\lambda}\log \frac{2}{\nu \mathfrak{h}_1\left(\frac{y}{\sqrt2}\right)}dy,\] if $\nu \mathfrak{h}_1\left(\frac{y}{\sqrt2}\right)=2$ has no solution. If this equation has solutions then, although we can find the area again we do not need to since $A_0$ will be even smaller in this case (see figure \ref{figure42}).
			
			\begin{figure}
				\centering
				\includegraphics[scale=0.2]{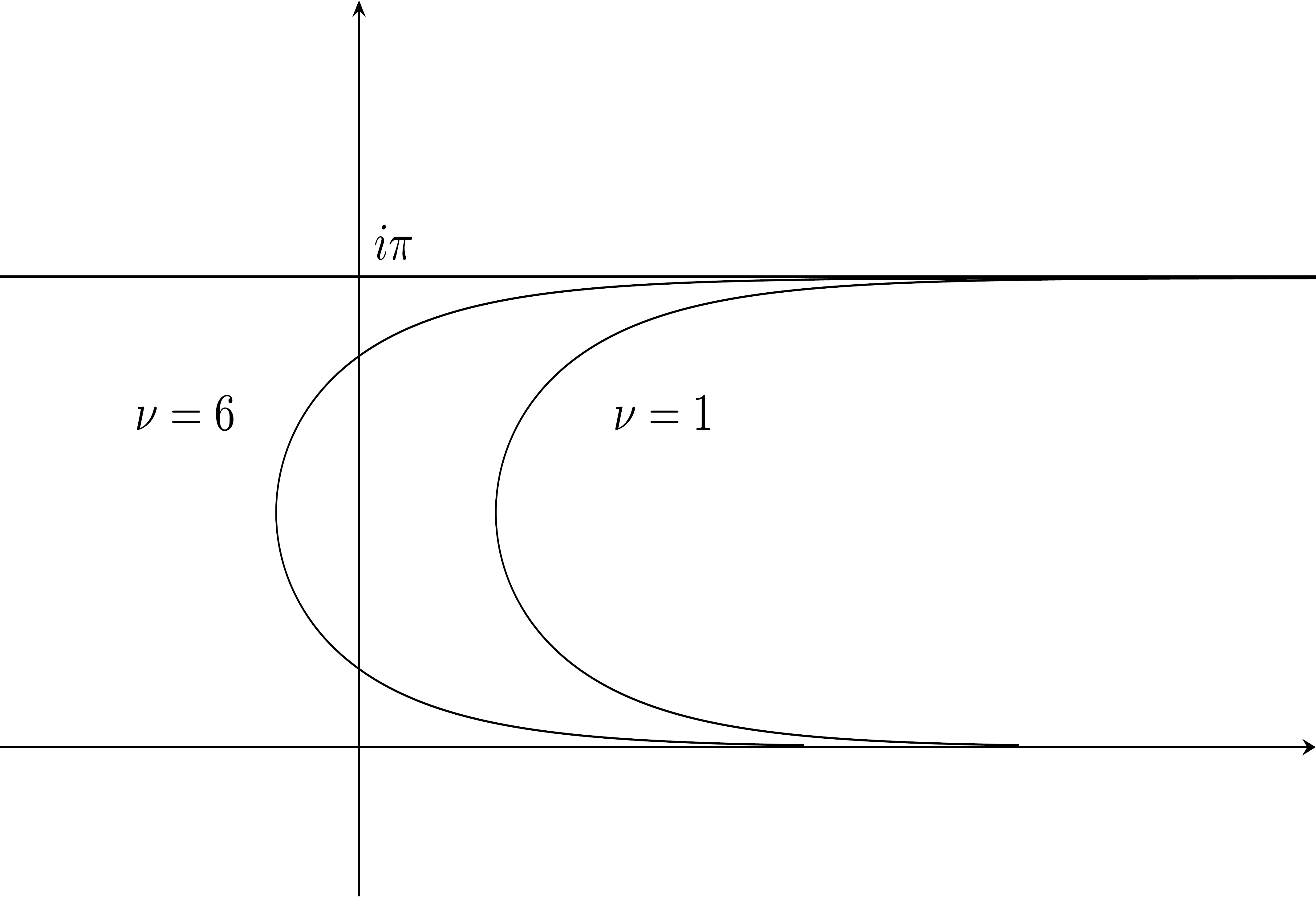}
				\caption{The curve $\gamma_0$ for the exponential map for two different values of $\nu$. The situation is similar with our maps as well.  }
				\label{figure42}
			\end{figure}
			Notice that because $(\mathfrak{h}_1,\mathfrak{h}_2,\mathfrak{h}_3)$ is always a point on the unit sphere we have that $$\mathfrak{h}_1(x,y)^2+\mathfrak{h}_2(x,y)^2=\sin^2 \theta|\mathfrak{h}(x,y)-\mathfrak{h}(0,0)|^2=\sin^2 \theta \left(\mathfrak{h}_1(x,y)^2+\mathfrak{h}_2(x,y)^2+(\mathfrak{h}_3(x,y)-1)^2\right),$$ where $\theta$ is the angle between the $x_3$-axis and the segment that connects $(0,0,1)$ with $(\mathfrak{h}_1,\mathfrak{h}_2,\mathfrak{h}_3)$. Taking $x=y$, the fact that $\mathfrak{h}$ is  bi-Lipschitz on $[-1,1]^2$ and noticing that $\theta \geq \pi/4$ gives us $|\mathfrak{h}_1(y)|\geq \frac{|y|}{\sqrt2L}.$ Hence because $\mathfrak{h}_1(y)>0$ for $y\in(0,1)$ we have \[A_0\leq \frac{2}{\lambda}\int_{0}^{\sqrt2}\log\left(\frac{4L}{\nu y}\right)dy,\] which is finite.
		\end{proof}

		\section{Proof of Theorem \ref{main}}\label{section 4}
		Having proved Theorem \ref{plane} we are now ready to proceed to our main theorem.
		
		First we prove some lemmas that we will later need. Note here that since we proved that the planes $x_1=\pm x_2$ are in the Julia set of $\mathcal{Z}_{\nu}$ we will also know that all their inverse images are in the Julia set. Those inverse images are again planes of the form $x_1=\pm x_2 + 2\lambda k$, where $k\in \mathbb{Z}$. Those planes partition $\mathbb{R}^3$ in square beams. Name $B_{(0,0)}$ the open rectangle beam that is the union of the two square beams that touch the $x_3$-axis and are in the half-space $x_2\leq x_1$. We can partition the space now in rectangle  beams that are translates of this $B_{(0,0)}$. Let us name them \[B_{(i,j)}=B_{(0,0)}+i(2\lambda,2\lambda,0)+j(\lambda,-\lambda,0), \hspace{2mm}i,j\in\mathbb{Z}.\]
		Note that the map $\mathcal{Z}_{\nu}$ is a homeomorphism in those rectangle beams. The next lemma is inspired by the one Misiurewicz used in his proof (compare \cite[Lemma 1]{Misiu}) and is the main reason that we need the scale factor $\lambda$ in our definition of the Zorich map. It will be convenient to introduce some notation. Let $p:\mathbb{R}^3\to \mathbb{R}^2$ be the projection map defined by $p(x_1,x_2,x_3)=(x_1,x_2)$. Also let $p_3(x)$, $x\in\mathbb{R}^3$ denote the third coordinate of $x$, in other words $p_3(x_1,x_2,x_3)=x_3$. 
		\begin{lemma}\label{dete}If $\lambda$ and $L$ are the numbers we used in the construction of the Zorich map and $\nu>0$ then
			\[\det\left(D\mathcal{Z}_{\nu}^n(x)\right)\geq {\left(\frac{\lambda }{L^5}\right)}^{n}\frac{1}{\lambda^{3}}\left|\left( p \circ \mathcal{Z}_{\nu}^{n}\right)(x)\right|^3\hspace{2mm}\text{a.e.}\]
			
		\end{lemma}
		\begin{proof}
			First note that $|p(\mathcal{Z}_{\nu}^n(x))|=\nu e^{p_3(\mathcal{Z}_{\nu}^{n-1})}\cdot |(p\circ h\circ p\circ\mathcal{Z}_{\nu}^{n-1})(x) |$. Also,	using the fact that $h(0)=(0,0,\lambda)$ and $h(x_1,x_2)=(h_1(x_1,x_2),h_2(x_1,x_2),h_3(x_1,x_2))$ is a Lipschitz map we have that \[|p(h(x))|=\left|p\left(h(x)-h(0)\right)\right|\leq|h(x)-h(0)|\leq L|x|. \]
			Hence
			\begin{align}\label{eq10}
				|p(\mathcal{Z}_{\nu}^n(x))|=&\nu e^{p_3(\mathcal{Z}_{\nu}^{n-1})}\left|\left(p\circ h\circ p \circ \mathcal{Z}_{\nu}^{n-1}\right)(x)\right|\nonumber\\\leq& L\nu e^{p_3(\mathcal{Z}_{\nu}^{n-1})}\left|\left( p \circ \mathcal{Z}_{\nu}^{n-1}\right)(x)\right|\leq\cdots\nonumber\\\leq &(\nu L)^{n}e^{p_3(\mathcal{Z}_{\nu}^{n-1})}e^{p_3(\mathcal{Z}_{\nu}^{n-2})}\cdots e^{x_3}\sqrt{h_1^2\Big (x_1,x_2\Big )+h_2^2\Big (x_1,x_2\Big )}\nonumber\\\leq& \lambda (\nu L)^{n}e^{p_3(\mathcal{Z}_{\nu}^{n-1})}e^{p_3(\mathcal{Z}_{\nu}^{n-2})}\cdots e^{x_3}.
			\end{align}
			From the chain rule we know that  \begin{equation}\label{chainrule}\det\left(D\mathcal{Z}_{\nu}^n(x)\right)=\prod_{k=0}^{n-1}\det\left(D\mathcal{Z}_{\nu}\left({\mathcal{Z}_{\nu}^{k}(x)}\right)\right).\end{equation}
			Now from the definition of $\mathcal{Z}_{\nu}$ we get \begin{equation}\label{eq1234}\det(D\mathcal{Z}_{\nu}(x))=\nu^3 e^{3x_3}\det H,\end{equation}
			where $$H=\begin{pmatrix}
				\frac{\partial h_1}{\partial x_1}(p(x))&\frac{\partial h_1}{\partial x_2}(p(x))&h_1(p(x))\\\\\frac{\partial h_2}{\partial x_1}(p(x))&\frac{\partial h_2}{\partial x_2}(p(x))&h_2(p(x))\\\\\frac{\partial h_3}{\partial x_1}(p(x))&\frac{\partial h_3}{\partial x_2}(p(x))&h_3(p(x))
			\end{pmatrix}.$$

			We now set \[A=\begin{pmatrix}
				\frac{\partial h_1}{\partial x_1}\left(p(x)\right)\\\\\frac{\partial h_2}{\partial x_1}\left(p(x)\right)\\\\\frac{\partial h_3}{\partial x_1}\left(p(x)\right)
			\end{pmatrix},\hspace{2mm} B=\begin{pmatrix}
				\frac{\partial h_1}{\partial x_2}\left(p(x)\right)\\\\\frac{\partial h_2}{\partial x_2}\left(p(x)\right)\\\\\frac{\partial h_3}{\partial x_2}\left(p(x)\right)\end{pmatrix},\hspace{2mm} C=\begin{pmatrix}
				h_1\left(p(x)\right)\\\\h_2\left(p(x)\right)\\\\h_3\left(p(x)\right)
			\end{pmatrix}.
			\]
			Recall now that from linear algebra, the determinant of a matrix equals the scalar triple product. This means that \[\det H=\big\langle A\times B, C\big\rangle,\] where $\langle \cdot,\cdot\rangle $ denotes the euclidean inner product.
			Since $\mathcal{Z_{\nu}}$ is sense preserving we will have that $\det H>0$ and since $A$ and $B$ are orthogonal to $C$ we will have that $A\times B$ is parallel to $C$. Remember that $|C|=\lambda$ so \begin{equation}\label{eq201}\det H= \lambda|A\times B| \Bigg\langle \frac{A\times B}{|A\times B|}, \frac{C}{|C|}\Bigg\rangle= \lambda |A\times B|.\end{equation}
			Now because $\mathfrak{h}$ is a locally bi-Lipschitz map we have  that \[\left|h\left(p(x)+tv\right)-h\left(p(x)\right)\right|\geq\frac{|tv|}{L},\] for all small $t>0$ where $v=(v_1,v_2)\in\mathbb{R}^2$. This implies that  \begin{equation}\label{eq123}\left|Dh\left(p(x)\right)(v)\right|\geq  \frac{|v|}{L}\end{equation} and if we set $v=\left(|B|,\frac{-\langle A,B\rangle}{|B|}\right)$ and square both sides we get \[\left| |B|A-\frac{\langle A,B\rangle}{|B|}B\right|^2\geq \frac{1}{L^2}\left(|B|^2+\frac{\langle A,B\rangle^2}{|B|^2}\right)\geq\frac{|B|^2}{L^2} .\] Simplifying we get 
			\[|A|^2|B|^2-\langle A,B\rangle^2\geq\frac{|B|^2}{L^2}. \]
			Now note that by elementary properties of the cross product $|A\times B|^2=|A|^2|B|^2-\langle A,B\rangle^2$ and thus $$|A\times B|^2\geq \frac{|B|^2}{L^2}\geq \frac{1}{L^4},$$ where the last inequality comes from (\ref{eq123}) for $v=(0,1)$. Hence, (\ref{eq201}) becomes $\det H\geq \frac{\lambda}{L^2}$.  
			Putting everything together in (\ref{eq1234}) we get  \begin{equation}\label{eqdeter}\det\left(D\mathcal{Z}_{\nu}(x)\right)\geq \nu^3 e^{3x_3}\frac{\lambda}{L^2},\hspace{2mm}\text{a.e.}\end{equation} 
			
			Hence, by (\ref{eq10}) and (\ref{eqdeter}) we have that \begin{align*}|p(\mathcal{Z}_{\nu}^n(x))|^3\leq&\lambda^3 (L\nu)^{3n}e^{3p_3(\mathcal{Z}_{\nu}^{n-1})}e^{3p_3(\mathcal{Z}_{\nu}^{n-2})}\cdots e^{3x_3}\\\leq&\frac{\lambda^3 L^{5n}}{\lambda^{n}}\det\left(D\mathcal{Z}_{\nu}\left(\mathcal{Z}_{\nu}^{n-1}(x)\right)\right)\cdots\det(D\mathcal{Z}_{\nu}(x)) \hspace{2mm}\text{a.e}.\end{align*} By rearranging and (\ref{chainrule}) now we get the desired inequality.
		\end{proof}
		The next lemma describes the behaviour of points near the $x_3$-axis under iteration.
		\begin{lemma}\label{upup}
			Let $\nu\lambda>\frac{1}{e}.$ 
			\begin{enumerate}
				
				\item[(a)] There are $\delta>0$ and $c>0$ such that if $x\in C_{\delta}$, where $C_{\delta}$ is the cylinder around the $x_3$-axis with $\delta$ radius, then $p_3(\mathcal{Z}_{\nu}(x))>p_3(x)+c$,
				\item[(b)]  For $\delta$ as in (a) and for every $x\in C_{\delta}$, with $p(x)\not=(0,0)$, there is an $n\in\mathbb{N}$ such that $\mathcal{Z}_{\nu}^n(x)\not\in C_{\delta}$.
			\end{enumerate}
			
		\end{lemma}
		\begin{proof}
			(a) We have if $h(x_1,x_2)=(h_1(x_1,x_2),h_2(x_1,x_2),h_3(x_1,x_2))$ then $p_3(\mathcal{Z}_{\nu}(x))=\nu e^{x_3}h_3(x_1,x_2)$. Now since $h(0,0)=(0,0,\lambda)$ and $h$ is continuous, for all $\varepsilon>0$ there is a disk $D=D(0,\delta)$ of radius $\delta>0$ on which  we have $h_3(x_1,x_2)>\lambda-\varepsilon$. Hence if $x=(x_1,x_2,x_3)\in C_{\delta}=D\times\mathbb{R}$, then \[p_3(\mathcal{Z}_{\nu}(x))=\nu e^{x_3}h_3(x_1,x_2)> \nu e^{p_3(x)}(\lambda-\varepsilon)\geq p_3(x)+1+\log \left(\nu(\lambda-\varepsilon)\right),\] where the last inequality follows by minimizing $\nu e^t(\lambda-\varepsilon)-t$. Now notice that since $\nu\lambda>\frac{1}{e}$ we can find a small enough $\varepsilon>0$ such that $\nu(\lambda-\varepsilon)>\frac{1}{e}$, which implies $1+\log(\nu(\lambda-\varepsilon))>0$. Hence, $p_3(\mathcal{Z}_{\nu}(x))>p_3(x)+c$ with $c= 1+\log(\nu(\lambda-\varepsilon))$.\\\\
			(b)   For a $\delta$ as in (a) and $\delta<\lambda$ now assume that there is a point $x\in C_{\delta}$ such that $p(x)\not=0$ and $\mathcal{Z}_{\nu}^n(x)\in C_{\delta}$ for all $n\in\mathbb{N}$. Then according to (a) we would have that $p_3(\mathcal{Z}_{\nu}^n(x))\to \infty$ when $n\to \infty$. We know that \begin{equation}\label{one}|(p\circ \mathcal{Z}_{\nu}^{n+1})(x)|=e^{p_3(\mathcal{Z}_{\nu}^n)}|\left(p\circ h\circ p\circ\mathcal{Z}_{\nu}^n\right)(x)|.\end{equation}
			Now its a simple geometric fact that for each $y\in[-\lambda,\lambda]^2$ \begin{equation}\label{two}|\left(p\circ h\right)(y)|=\sin \theta \left|h(y)-h(0)\right|,\end{equation} where $\theta$ is the angle between the line segment joining the point $h(y)$, on the sphere, with the point $(0,0,\lambda)$ and the $x_3$-axis. Moreover, $\theta\geq \pi/4$ for all such $y$. Also by the fact that $h$ is a bi-Lipschitz function we have that $|h(y)-h(0)|\geq \frac{|y|}{L}.$

			Now taking $y=p\left(\mathcal{Z}_{\nu}^n(x)\right)$ and combining this with (\ref{one}) and (\ref{two}) implies that \[|(p\circ \mathcal{Z}_{\nu}^{n+1})(x)|\geq \frac{e^{p_3(\mathcal{Z}^n(x))}}{\sqrt2 L}|(p\circ\mathcal{Z}_{\nu}^n)(x)|.\] Thus, since $p_3(\mathcal{Z}_{\nu}^n(x))\to \infty$, for all large enough $n$  we can say that \[|(p\circ\mathcal{Z}_{\nu}^{n+1})(x)|\geq2|(p\circ\mathcal{Z}_{\nu}^n)(x)|.\] This of course contradicts the fact that $\mathcal{Z}^n(x)\in C_{\delta}$ for all $n\in\mathbb{N}$.
		\end{proof}
		The next  lemmas describe how sets of positive measure behave under iteration by the Zorich map assuming that their iterates never cross the planes that we already know  belong in the Julia set.
		\begin{lemma}\label{metro}
			Assume that $\lambda>L^{5}$.	Let $V\subset\mathbb{R}^3$ be a connected set with $m(V)>0$ and whose iterates do  not intersect any of the planes $x_1=\pm x_2 + 2\lambda k$, where $k\in\mathbb{Z}$. Suppose also that there is sequence of integers $n_j>0$ with $\mathcal{Z}_{\nu}^{n_j}(V)\cap C_a=\emptyset$, where $C_a$ is a cylinder around the $x_3$-axis of any radius $a>0$. Then $m(\mathcal{Z}_{\nu}^{n_j}(V))\to\infty$ as $n_j\to \infty$,  where $m$ is the 3 dimensional Lebesgue measure.
		\end{lemma}
		\begin{proof}
			
			Since  $\mathcal{Z}_{\nu}^{n_j}(V)$ stays out of the cylinder $C_a$ we have that, for $x\in V$ \[|(p\circ \mathcal{Z}_{\nu}^{n_j})(x)|>a.\] By using Lemma \ref{dete} we will now have that \[\det\left(D\mathcal{Z}_{\nu}^{n_j}\right)\geq{\left(\frac{\lambda }{L^5}\right)}^{n_j} \cdot \frac{a^3}{\lambda^{3}}  \hspace{2mm} \text{a.e. on} \hspace{1mm} V.\] 
			Since all of the iterates of $V$ do not intersect any of the planes $x_1=\pm x_2 + 2\lambda k$, where $k\in\mathbb{Z}$  and since the Zorich map is a homeomorphism in the square beams that remain if we remove those planes we will have that $\mathcal{Z}^n$ is a homeomorphism in $V$. Hence, for all $n_j\in\mathbb{N}$ we will have  that \[m(\mathcal{Z}_{\nu}^{n_j}(V))=\int_{V}\left|\det\left(D\mathcal{Z}_{\nu}^{n_j}\right)\right|dm\geq{\left(\frac{\lambda }{L^5}\right)}^{n_j}{\left(\frac{a^3}{\lambda^{3}}\right)} m(V),\] which tends to infinity as $j\to \infty$ since  $\lambda>{L^{5}}$.
		\end{proof}
		For our next lemma let us assume  that our Zorich map sends the beam $B_{(0,0)}$ to the half space $x_2\leq x_1$. The other alternative is mapping it to the half space $x_1\leq x_2$ but the methods work in a very similar way with minor modifications.

		Consider the inverse image under $\mathcal{Z}_{\nu}$ of the boundary of $B_{(0,0)}$ that lies in the interior of $B_{(0,0)}$. This inverse image will be some surface which we will call $S_0$. In Figure \ref{figure1} we have drawn the $x_1$$x_2$ plane and the rectangle beams $B_{(0,0)}$ and $B_{(0,-1)}$. Take now the planes $P_1:x_2=x_1-4\lambda$, $P_2:x_2=-x_1+4\lambda$, $P_3: x_2=-x_1-4\lambda$ and consider the rectangle beam they define together with the plane $x_1=x_2$. Let us now take the boundary of this beam, without the part that belongs to $x_1=x_2$, and name it $L_1$. Consider  now the inverse image of $L_1$ that lies inside $B_{(0,0)}$. This image is a surface, let us call it $S_1$. We can now do the same with the planes $P_4:x_2=x_1-6\lambda$, $P_5:x_2=-x_1+6\lambda$, $P_6: x_2=-x_1-6\lambda$ and get the boundary of the beam they define, which we call $L_2$ and then the surface we get by taking the inverse image which we call $S_2$. If we continue with this construction we get a sequence of surfaces, $S_0$, $S_1$, $S_2$, $S_3$, $\cdots$  inside $B_{(0,0)}$. Each of those surfaces lies above its previous starting with $S_0$. We can also construct similar surfaces $K_0, K_1, K_2, \cdots$ inside the beam $B_{(0,-1)}$  by taking inverse images of the corresponding boundaries $\partial B_{(0,-1)}$, $R_1$, $R_2$, $\cdots$ (see Figure \ref{figure1}). Moreover, we construct similar surfaces in all the other rectangle beams $B_{(i,j)}$, that partition the space,  depending on which half-space the beam is mapped to under the Zorich map. Let us denote by $\mathcal{S}$ the union of all those surfaces.
		
		We will show that the space between $S_n$ and $S_{n+1}$ (similarly between $K_n$ and $K_{n+1}$) is of finite volume and is decreasing as $n$ increases.

		\begin{figure}[h]
			\centering
			\includegraphics{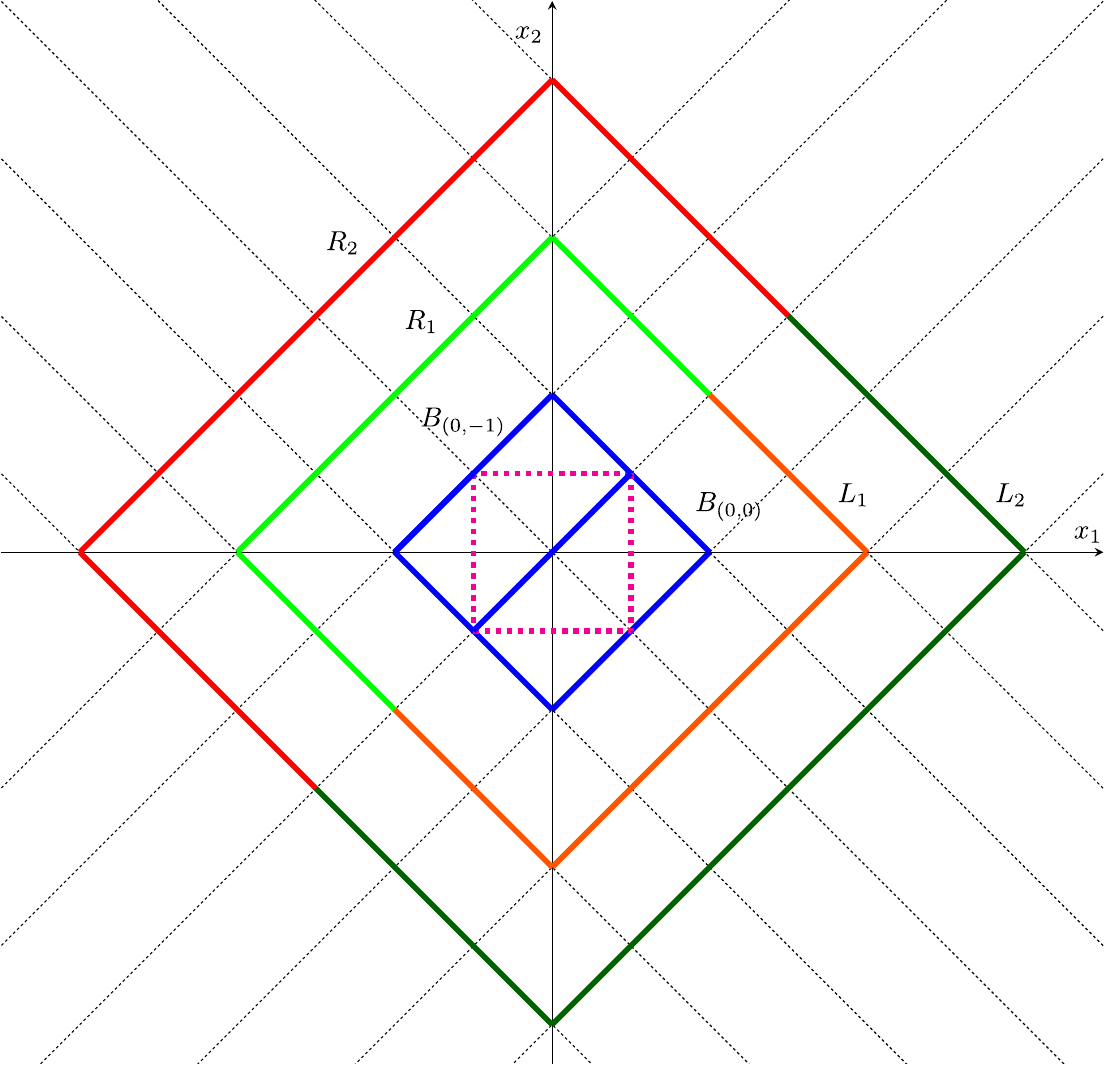}
			
			\caption{\small{The $x_1x_2$ plane. In pink is the initial square we used to define our Zorich map. In blue the rectangle  beams $B_{(0,0)}$ and $B_{(0,-1)}$ while in orange, dark green, green and red are the sets $L_1$, $L_2$, $R_1$ and $R_2$ respectively.}}
			\label{figure1}
		\end{figure}
		
		\begin{lemma}\label{surfacelemma}
			Let $I_n$ be the volume that the surface $S_n$ encloses together with the plane  $x_3=0$ and inside the beam $B_{(0,0)}$. Then $I_n$ is finite for all $n\in\mathbb{N}$. Furthermore, if $T_n:=I_{n+1}-I_n$ is the volume between $S_n$ and $S_{n+1}$ then $T_n$ is a decreasing sequence.  
		\end{lemma}
		\begin{proof}
			Let us first find an implicit equation that describes each of these surfaces. We work on $B_{(0,0)}$ but the same can be done  on all other rectangle beams. Let us split $B_{(0,0)}$ in three different beams whose cross-sections with the $x_1x_2$ plane are the sets $$Q_1:=h^{-1}\left(\{(x_1,x_2,x_3)\in S(0,\lambda):0\leq x_2\leq x_1\}\right),$$ $$Q_2:=h^{-1}\left(\{(x_1,x_2,x_3)\in S(0,\lambda): x_2\leq 0, x_1\geq 0\}\right)$$ and $$Q_3:=h^{-1}\left(\{(x_1,x_2,x_3)\in S(0,\lambda):x_2\leq x_1\leq 0\}\right),$$ where $S(0,\lambda)$ the sphere of centre $0$ and radius $\lambda$. In the  beam corresponding to the first cross section, meaning $Q_1\times \mathbb{R}$, the points on the surface $S_n$ satisfy \[\nu e^{x_3}h_1(x_1,x_2)=-\nu e^{x_3}h_2(x_1,x_2)+2(n+1)\lambda.\]On the beam $Q_2\times \mathbb{R}$ the surface points satisfy \[\nu e^{x_3}h_1(x_1,x_2)=\nu e^{x_3}h_2(x_1,x_2)+2(n+1)\lambda,\] while on the beam $Q_3\times \mathbb{R}$ that corresponds to the last cross section the points satisfy \[\nu e^{x_3}h_1(x_1,x_2)=-\nu e^{x_3}h_2(x_1,x_2)-2(n+1)\lambda.\]  Suppose now that the surfaces $S_n$ do not intersect the plane $x_3=0$. Hence the volume that the surface $S_n$ encloses together with the plane  $x_3=0$ and inside the beam $B_{(0,0)}$ is given by the integrals \begin{multline*}
				I_n=\int\int_{Q_1}\log \frac{2(n+1)\lambda}{\nu\left(h_2(x_1,x_2)+h_1(x_1,x_2)\right)}dx_2dx_1+\int\int_{Q_2}\log \frac{-2(n+1)\lambda}{\nu\left(h_1(x_1,x_2)-h_2(x_1,x_2)\right)}dx_2dx_1\\+\int\int_{Q_3}\log \frac{-2(n+1)\lambda}{\nu\left(h_2(x_1,x_2)+h_1(x_1,x_2)\right)}dx_2dx_1.
			\end{multline*}
			It is not so hard to prove now that each of these integrals is finite since  each one of them is an integral of a bounded function except at a neighbourhood of $(0,0)$ and the points where $h_1\pm h_2 =0$ which are $(0,-2\lambda)$, $(2\lambda,0)$. So in order to show that this sum is finite, it is enough to consider the integrals only in neighbourhoods of those points. On the other hand, when those surfaces $S_n$ intersect the plane $x_3=0$ the volumes are no longer given by the above integrals (the set where we integrate will change) but again we only have to consider them at a neighbourhood of $(0,0)$ as well as  $(0,-2\lambda)$, $(2\lambda,0)$ so that is what we do next. In fact, those volumes in that case are even smaller. 
			
			We will only treat here the second integral around an $\varepsilon$-neighbourhood of $(0,0)$ and the rest follows similarly. So we are looking at the integral
			\[\int\int_{Q_2\cap B(0,\varepsilon)}\log \frac{-2(n+1)\lambda}{\nu\left(h_2(x_1,x_2)-h_1(x_1,x_2)\right)}dx_2dx_1,\] where $B(0,\varepsilon)$ is a ball centred at $0$ with radius $\varepsilon$.
			Equivalently, we want to show that the integral \begin{equation}\label{integral}\int\int_{Q_2\cap B(0,\varepsilon)}-\log \left(h_1(x_1,x_2)-h_2(x_1,x_2)\right)dx_2dx_1\end{equation} is finite. Now because $h$ is a bi-Lipschitz map and because $h_2(x_1,x_2)\leq 0$ and $h_1(x_1,x_2)\geq0$ in the set we are integrating we have that \[(h_1-h_2)^2=h_1^2+h_2^2+2h_1(-h_2)\geq h_1^2+h_2^2=\sin^2 \theta (h_1^2+h_2^2+(h_3-\lambda)^2)\geq \frac{c_{\varepsilon}}{L^2}(x_1^2+x_2^2),\]where $\theta$ is the angle between the $x_3$ axis and the segment that connects $(0,0,\lambda)$ with $(h_1,h_2,h_3)$ and $c_{\varepsilon}>0$ a constant that depends only on $\varepsilon$. Hence \[|h_1(x_1,x_2)-h_2(x_1,x_2)|\geq \frac{\sqrt{c_\varepsilon}}{L}\sqrt{x_1^2+x_2^2}.\]  Now since $h_1-h_2\geq0$ in the set we are integrating, we will have that \[\int\int_{Q_2\cap B(0,\varepsilon)}\log \left(h_1(x_1,x_2)-h_2(x_1,x_2)\right)dx_2dx_1\geq \int\int_{Q_2\cap B(0,\varepsilon)}\log \left(\frac{\sqrt{c_\varepsilon}}{L}\sqrt{x_1^2+x_2^2}\right)dx_2dx_1.\]Now since the last integral is finite we will have that the integral (\ref{integral}) is also finite.
			
			Finally, let us show that the sequence $T_n$ is a decreasing one. Indeed, \[T_n=I_{n+1}-I_n=\lambda^2\log\frac{n+2}{n+1}+\lambda^2 \log\frac{n+2}{n+1}+2\lambda^2\log\frac{n+2}{n+1}=4\lambda^2\log\frac{n+2}{n+1},\]
			which can be easily seen to be a decreasing sequence.
		\end{proof}
		\begin{lemma}\label{axis}
			Assume $\lambda>L^{5}$ and let $V$ be a connected subset of $\mathbb{R}^3$ with $m(V)>0$ and such that $\mathcal{Z}_{\nu}^n(V)$ does not intersect any of the planes $x_1=\pm x_2 +2\lambda k$, where $k\in\mathbb{Z}$ for all $n\in\mathbb{N}$. Then $\mathcal{Z}_{\nu}^n(V)$ visits infinitely often one of the two rectangle beams $B_{(0,0)}$, $B_{(0,-1)}$, that have the $x_3$-axis in their boundary.
		\end{lemma}
		\begin{proof}
			Consider the iterates $V_i=\mathcal{Z}_{\nu}^i(V)$ of the set $V$. The sets $V_i$ stay always inside one of the rectangle beams by assumption and also they cannot intersect any of the surfaces in $\mathcal{S}$ that are in those beams since if they did on the next iterate they would intersect the boundary of one of the beams.  Suppose now that we can find a $N\in\mathbb{N}$  such that $V_{i}\not\in B_{(0,0)}\cup B_{(0,-1)}$ for all $i>N$. Then by Lemma \ref{metro} we have that $m(V_{i})\to\infty$. This implies that our sets $V_{i}$ cannot lie between any two of the surfaces in $\mathcal{S}$, for all large $i$ since there is finite volume between them. Thus $V_{i}$ stays below the lowest surface in the relevant rectangle beam for all $i>N_1>N$, where $N_1\in\mathbb{N}$. This is a contradiction since being below that surface implies that $V_{i+1}$ is in either $B_{(0,0)}$ or $B_{(0,-1)}$.
		\end{proof}
		
		The next lemma tells us that when a set remains in $B_{(0,0)}$ under iteration by $\mathcal{Z}_{\nu}$ then we can find points with large $x_3$ coordinate in its iterates.
		\begin{lemma}\label{newlemma}
			Assume $\lambda>L^{5}$, $\nu>\frac{1}{\lambda e}$ and let $V$ be a connected set of $\mathbb{R}^3$ with $m(V)>0$ and such that $\mathcal{Z}_{\nu}^n(V)$ does not intersect any of the planes $x_1=\pm x_2 +2\lambda k$, where $k\in\mathbb{Z}$ for all $n\in\mathbb{N}$. Suppose that there is an $N_0\in \mathbb{N}$ such that  $\mathcal{Z}_{\nu}^n(V)\subset B_{(0,0)}$, for all $n>N_0$. Then for all $M>0$ and $\varepsilon>0$ there is some $n_0>N_0$ and a point $x\in \mathcal{Z}_{\nu}^{n_0}(V)$ such that $p_3(x)>M$ and $d(x,x_3\text{-axis})<\varepsilon$, where $d$ is the euclidean distance.
		\end{lemma}
		\begin{proof}
			Either an $\varepsilon_0>0$  exists such that $d(\mathcal{Z}_{\nu}^n(V),x_3\text{-axis})>\varepsilon_0$ for all $n>N_0$ or such an $\varepsilon_0$ does not exist. In the first case we know from Lemma \ref{metro} that $m(\mathcal{Z}_{\nu}^n(V))\to\infty$. We also know that since $\mathcal{Z}_{\nu}^n(V)\in B_{(0,0)}$, for all $n>N_0$, $\mathcal{Z}_{\nu}^n(V)$ must, by Lemma \ref{surfacelemma},  lie below the surface $S_0$ . Since $m(\mathcal{Z}_{\nu}^n(V))\to\infty$ then it must be true that for all $M_1>0$ there is a $n_0$ and a point $z_0$ in $\mathcal{Z}_{\nu}^{n_0}(V)$ with $p_3(z_0)<-M_1$. The pre-image of that point inside $B_{(0,0)}$ is a point $z^{(1)}=\left(z_1^{(1)},z_2^{(1)},z_3^{(1)}\right)$ for which $\nu e^{z_3^{(1)}}h_3\left(z_1^{(1)},z_2^{(1)}\right)<-M_1$ and $h_3\left(z_1^{(1)},z_2^{(1)}\right)<0$. But since $h_3>-\lambda$ we have that \[e^{z_3^{(1)}}>\frac{M_1}{\nu \lambda}\Rightarrow z_3^{(1)}>\log\frac{ M_1}{\nu \lambda}.\] If we now take the pre-image in $B_{(0,0)}$ of that point, $z^{(2)}=\left(z_1^{(2)},z_2^{(2)},z_3^{(2)}\right)$ then $0<h_3\left(z_1^{(2)},z_2^{(2)}\right)<\lambda$ and \[ \nu e^{z_3^{(2)}}h_3\left(z_1^{(2)},z_2^{(2)}\right)=z_3^{(1)}>\log \frac{M_1}{\nu\lambda},\]which implies then $$z_3^{(2)}>\log\frac{\log \frac{M_1}{\nu\lambda}}{\nu\lambda}.$$ Thus we have shown that for any $M>0$ there is a point $z^{(2)}$ in $\mathcal{Z}_{\nu}^{n_0-2}(V)$  for which $z_3^{(2)}>M$ and also $\left|z_1^{(2)}\right|,\left|z_2^{(2)}\right|< \lambda$. This leads to a contradiction. To see why, note that by our assumptions $\mathcal{Z}_{\nu}^n(V)$ is $\varepsilon_0$ away from the $x_3$-axis and below the surface $S_0$ and thus all of its points, which are also inside the initial square beam $[-\lambda,\lambda]^2\times \mathbb{R}$ (pink in figure \ref{figure1}), have a bounded $x_3$ coordinate. This is true because the surface $S_0$ together with any cylinder around the $x_3$ axis and the plane $x_3=0$ enclose a set inside $[-\lambda,\lambda]^2\times \mathbb{R}$ and outside of the cylinder whose closure is compact. 
			
			For the second case now, where such an $\varepsilon_0$ does not exist, then there is a sequence $w_k\in\cup_{n>N_0}\mathcal{Z}_{\nu}^n(V)$ with $d(w_k,x_3\text{-axis})\to 0$. If  $p_3(w_k)\to\infty$ we are done. If on the other hand  $p_3(w_k)\to-\infty$ then $y_k:=\mathcal{Z}_{\nu}(w_k)\to (0,0,0)$ and thus $\mathcal{Z}_{\nu}^n(y_k)\to \mathcal{Z}_{\nu}^n(0)=(0,0, E_{\nu\lambda}^n(0))$, where $E_{\nu\lambda}^n(0))$ converges to $\infty$ and again we are done. The remaining case to consider is when there is a subsequence $w_{k_i}$ converging to some point  $(0,0,a)$ in the $x_3$-axis. By relabelling we may assume that $w_k\to (0,0,a)$. Now choose an $N>0$ such that $E_{\nu\lambda}^N(a)>M$, where $E_{\nu\lambda}$ denotes the map $x\mapsto\nu\lambda e^x$. By continuity of $\mathcal{Z}_{\nu}^N$, for all $\varepsilon>0$ we may find a $\delta$  such that if $|w_k-(0,0,a)|\leq \delta$, then \[\left|\mathcal{Z}_{\nu}^N(w_k)-\mathcal{Z}_{\nu}^N(0,0,a)\right|=\left|\mathcal{Z}^N_{\nu}(w_k)-\left(0,0,E_{\nu\lambda}^N(a)\right)\right|\leq \varepsilon.\] Hence, if we choose $\varepsilon$ small enough we get that $x_3\left(\mathcal{Z}^N_{\nu}(w_k)\right)>M$ and $ \mathcal{Z}^N_{\nu}(w_k)$ is within $\varepsilon$ distance from the $z$-axis, when $k$ is large enough.
		\end{proof}
		\begin{lemma}\label{lipschitz}
			Let $y_1,y_2\in B(0,r)$, where $r>0$. Then  for all $n\in \mathbb{N}$ it is true that \[|\mathcal{Z}_{\nu}^n(y_1)-\mathcal{Z}_{\nu}^n(y_2)|\leq \left(\frac{\max\{L,\lambda\}}{\lambda}\right)^n E_{\nu\lambda}(r)\cdots E_{\nu\lambda}^n(r)|y_1-y_2|,\] where $E_{\nu\lambda}$ denotes the exponential map $x\mapsto\nu\lambda e^x$ and $L$ is the bi-Lipschitz constant we used in the construction of the Zorich maps.
		\end{lemma}
		\begin{proof}
			The Zorich map is absolutely continuous on any line segment since it is  locally Lipschitz. Using the fundamental theorem of calculus for the Lebesgue integral now, it is not too hard to prove  that a version of  the finite increment theorem, see \cite[10.4.1, Theorem 1]{Zorich2019}, is true for such functions. Specifically, we can prove that \[|\mathcal{Z}_{\nu}^n(y_1)-\mathcal{Z}_{\nu}^n(y_2)|\leq \esssup_{x\in \gamma}|D\mathcal{Z}_{\nu}^n(x)||y_1-y_2|,\] where $\gamma$ is the line segment that connects $y_1$ to $y_2$.	Remember that $|Df|$ denotes the operator norm of the total derivative. Hence, by the chain rule and elementary properties of linear maps we have that \[|D\mathcal{Z}_{\nu}^n(x)|\leq |D\mathcal{Z}_{\nu}\left(\mathcal{Z}_{\nu}^{n-1}(x)\right)|\cdots |D\mathcal{Z}_{\nu}(x)|.\] Hence by the above inequalities and because $y_1,y_2 \in B(0,r)$ we have that
			\begin{align}\label{ineq}|\mathcal{Z}_{\nu}^n(y_1)-\mathcal{Z}_{\nu}^n(y_2)|\leq&\nonumber \esssup_{x\in \gamma}|D\mathcal{Z}_{\nu}\left(\mathcal{Z}_{\nu}^{n-1}(x)\right)|\cdots \esssup_{x\in \gamma}|D\mathcal{Z}_{\nu}(x)||y_1-y_2|\\\leq& \esssup_{x\in B(0,r)}|D\mathcal{Z}_{\nu}\left(\mathcal{Z}_{\nu}^{n-1}(x)\right)|\cdots \esssup_{x\in B(0,r)}|D\mathcal{Z}_{\nu}(x)||y_1-y_2| .\end{align}

			We also know that $D\mathcal{Z}_{\nu}(x)= e^{x_3}D\mathcal{Z}_{\nu}(x_1,x_2,0)$.  Moreover, we will prove that \[|D\mathcal{Z}_{\nu}(x_1,x_2,0)|\leq \nu\max\{L,\lambda\}.\] Indeed, let $u=(u_1,u_2,u_3)\in\mathbb{R}^3$ then

			\begin{align*}|D\mathcal{Z}_{\nu}(x_1,x_2,0)|^2=\sup_{|u|=1}|D\mathcal{Z}_{\nu}(x_1,x_2,0)(u)|^2= \nu^2\sup_{|u|=1}|u_1A+u_2B+u_3C|^2,  \end{align*}
			where $A,B, C$ are as in the proof of Lemma \ref{dete}.
			Remember now that $C$ is orthogonal to $A$ and $B$ and thus the  above equation becomes
			\begin{align*}
				|D\mathcal{Z}_{\nu}(x_1,x_2,0)|^2&=\nu^2\sup_{|u|=1}\left(|u_1A+u_2B|^2+|u_3C|^2\right)\\&=\nu^2\sup_{|u|=1}\left(\left|Dh(x)(u_1,u_2)\right|^2+|u_3|^2|C|^2\right)\\&\leq\nu^2\sup_{|u|=1}\left(L^2|(u_1,u_2)|^2+\lambda^2|u_3|^2\right)\\&\leq\nu^2\max\{L^2,\lambda^2\},
			\end{align*}
			where we have used the fact that $h$ is locally bi-Lipschitz.

			Hence, $|D\mathcal{Z}_{\nu}(x)|\leq \nu\max\{L,\lambda\}e^{x_3}$ and (\ref{ineq}) becomes \begin{align*}|\mathcal{Z}_{\nu}^n(y_1)-\mathcal{Z}_{\nu}^n(y_2)|\leq& \nu^n\max\{L,\lambda\}^n\sup_{x\in B(0,r)}e^{p_3(\mathcal{Z}_{\nu}^{n-1}(x))}\cdots \sup_{x\in B(0,r)}e^{p_3(x)}|y_1-y_2|\\=&\left(\frac{\max\{L,\lambda\}}{\lambda}\right)^n E_{\nu\lambda}(r)\cdots E_{\nu\lambda}^n(r)|y_1-y_2|,\end{align*}
			where we have used the fact that $\nu\lambda\sup_{x\in B(0,r)}e^{p_3(\mathcal{Z}_{\nu}^{n}(x))}=E_{\nu\lambda}^{n+1}(r)$ which can be easily proved by induction on $n$.
		\end{proof}
		\begin{proof}[Proof of Theorem \ref{main}]
			First, let us note that by assumption $\nu>\sqrt{\frac{2L}{\lambda}}>\frac{1}{\lambda e}$. Let $V$ be any open and connected set of $\mathbb{R}^3$. We want to show that $\mathcal{Z}_{\nu}^n(V)$ intersects one of the planes that belong to the Julia set for some $n$ and thus $V$ itself intersects the Julia set. Assume that this does not happen. By Lemma \ref{axis} now we can consider two cases \\\\\textbf{\underline{First Case}}\\
			Suppose first that the sequence of iterates $\mathcal{Z}_{\nu}^n(V)$ does not eventually stay inside the square beam $B_{(0,0)}\cup B_{(0,-1)}$ but it also visits their complement infinitely often. Then we can find a subsequence $n_j$ such that $\mathcal{Z}_{\nu}^{n_j}(V)\in B_{(0,0)}\cup B_{(0,-1)}$ and $\mathcal{Z}_{\nu}^{n_j+1}(V)\in B_{(k,l)}$ for some $(k,l)\not =(0,0),(0,-1)$. Without loss of generality we may assume that $\mathcal{Z}_{\nu}^{n_j}(V)\in B_{(0,0)}$. 
			
			Consider now the sets \[V_{n_j}^+=\{x\in V: |(p\circ \mathcal{Z}_{\nu}^{n_j})(x)|\geq \frac{\lambda}{2}\}\] and \[V_{n_j}^-=\{x\in V: |(p\circ \mathcal{Z}_{\nu}^{n_j})(x)|<\frac{\lambda}{2}\},\]($\lambda$ is the scale factor by which we scaled up the initial square). Notice that $V=V_{n_j}^+\cup V_{n_j}^-$. Since $\mathcal{Z}_{\nu}^{n_j+1}(V)$ is outside of $B_{(0,0)}\cup B_{(0,-1)}$ we will have that $\mathcal{Z}_{\nu}^{n_j}(V)$ lies between two "level surfaces" $S_k$ and $S_{k+1}$ but we know, from Lemma \ref{surfacelemma}, that those surfaces enclose $\leq M_0$ volume between them, where $M_0$ is a constant. Thus $m(\mathcal{Z}_{\nu}^{n_j}(V))\leq M_0$, where $m$ denotes the Lebesgue measure. Also by Lemma \ref{dete} it is true that for almost all points in $V_{n_j}^+$ 
			\[\det\left(D\mathcal{Z}_{\nu}^{n_j}\right)\geq{\left(\frac{\lambda }{L^5}\right)}^{n_j}\frac{1}{\lambda^{3}}|(p\circ \mathcal{Z}_{\nu}^{n_j})(x)|^3\geq{\left(\frac{\lambda }{L^5}\right)}^{n_j}\frac{1}{8}.
			\]
			Hence, because $\mathcal{Z}_{\nu}^{n_j}$ is a homeomorphism in $V$ we have that \[M_0\geq m(\mathcal{Z}_{\nu}^{n_j}(V_{n_j}^+))=\int_{V_{n_j}^+} |\det(D\mathcal{Z}_{\nu}^{n_j})|dm\geq{\left(\frac{\lambda }{L^5}\right)}^{n_j}\frac{1}{8}\cdot m(V_{n_j}^+).\]
			This implies that $m(V_{n_j}^+)\to 0$ as $n_j\to \infty$. On the other hand, the set $\mathcal{Z}_{\nu}^{n_j}(V_{n_j}^-)$ is inside the initial square beam $[-\lambda,\lambda]^2\times \mathbb{R}$. Thus $\mathcal{Z}_{\nu}^{n_j+1}(V_{n_j}^-)$ lies in the half space $x_3>0$ and outside the square beam $B_{(0,0)}\cup B_{(0,-1)}$. This same set also lies between some level surfaces or below all of them. Moreover, as we proved in Lemma \ref{surfacelemma} the volume enclosed by those successive surfaces and the volume enclosed by the first one and the plane $x_3=0$ is smaller than some constant $M_0$. Thus $m (\mathcal{Z}_{\nu}^{n_j+1}(V_{n_j}^-))\leq M_0$. By arguing the same way as before now we get that \[M_0\geq m(\mathcal{Z}_{\nu}^{n_j+1}(V_{n_j}^-))=\int_{V_{n_j}^-} |\det(D\mathcal{Z}_{\nu}^{n_j+1})|dm\geq{\left(\frac{\lambda }{L^5}\right)}^{n_j+1}\frac{1}{8}\cdot m(V_{n_j}^-).\] 
			This again implies that $ m(V_{n_j}^-)\to 0$ as $n_j\to \infty$. But this is a contradiction since $m(V)=m(V_{n_j}^-)+m(V_{n_j}^+)$.
			\\\\\textbf{\underline{Second Case}}\\
			Suppose now that $\mathcal{Z}_{\nu}^n(V)\in B_{(0,0)}\cup B_{(0,-1)}$, for all $n>N_0$. Observe that either $\mathcal{Z}_{\nu}(B_{(0,0)})\subset\{(x_1,x_2,x_3):x_2\leq x_1\}$ or $\mathcal{Z}_{\nu}(B_{(0,0)})\subset\{(x_1,x_2,x_3):x_2\geq x_1\}$.  In the first case  $\mathcal{Z}_{\nu}^n(V)$ stays in $B_{(0,0)}$  for all large $n$ or it stays in  $B_{(0,-1)}$  while in the second it alternates between $B_{(0,0)}$ and $B_{(0,-1)}$. For simplicity we will assume that the first case holds and thus $\mathcal{Z}_{\nu}^n(V)\in B_{(0,0)}$, for all $n>N_0$.

			Let us now consider the inverse image under $\mathcal{Z}_{\nu}$ of the boundary of $B_{(0,0)}$, that lies inside $B_{(0,0)}$, namely the surface $S_0$ we had in the proof of Lemma \ref{surfacelemma}. Remember that this surface is defined as $S_0:=\{(x_1,x_2,x_3)\in B_{(0,0)}:x_3=f(x_1,x_2)\}$, where $f$ is continuous on $B_{(0,0)}\cap \{x_3=0\}$ and extends continuously on the boundary of this set except at the points $(0,0)$, $(2\lambda,0)$, $(0,-2\lambda)$ where $f\to \infty$.  Notice then that all the iterates $\mathcal{Z}_{\nu}^n(V)$ stay below the surface $S_0$.

			Consider now a plane $x_3=c$, with $c=E^{N}_{\nu\lambda}(0)-\lambda$ and $N$ so large that this plane intersects this surface $S_0$ and also \begin{equation}\label{key}  (c+\lambda)^{\log (c+\lambda)+1}e^{c+\lambda}\nu^2\lambda^2 e^{-\frac{\nu\lambda e^c}{2}}\leq \lambda.\end{equation}

			We define now sets $A_1$, $A_2$ and $A_3$ as follows:
			\begin{itemize}
				\item $A_1:=\{(x_1,x_2,x_3)\in B_{(0,0)}:c<x_3<f(x_1,x_2)\hspace{2mm}\text{and}\hspace{2mm} (x_1,x_2)\hspace{2mm}\text{in a neighbourhood of}\hspace{2mm}(0,0)\}$.
				\item $A_2:=\{(x_1,x_2,x_3)\in B_{(0,0)}:c<x_3<f(x_1,x_2)\hspace{2mm}\text{and}\hspace{2mm} (x_1,x_2)\hspace{2mm}\text{in a neighbourhood of}\hspace{2mm}(2\lambda,0)\}$.
				\item $A_3:=\{(x_1,x_2,x_3)\in B_{(0,0)}:c<x_3<f(x_1,x_2)\hspace{2mm}\text{and}\hspace{2mm} (x_1,x_2)\hspace{2mm}\text{in a neighbourhood of}\hspace{2mm}(0,-2\lambda)\}$.
			\end{itemize}

			We now have that:
			\begin{enumerate}[label=(\roman*)]
				\item All those sets lie below (in terms of $x_3$ coordinate) the surface $S_0$.
				By the definitions and Lemma \ref{surfacelemma} it is  easy to see that the sets $A_1$, $A_2$ and $A_3$ are also of finite Lebesgue measure.
				\item $\mathcal{Z_{\nu}}(A_2\cup A_3)\subset \{(x_1,x_2,x_3)\in\mathbb{R}^3:x_3< -\frac{\nu \lambda e^c}{2}\}$ and thus $\mathcal{Z}^2_{\nu}(A_2\cup A_3)\subset B(0,\delta), $ where $\delta= \nu\lambda e^{-\frac{\nu \lambda e^c}{2}}$. Note that $\delta<\nu\lambda=E_{\nu\lambda}(0)$.
				\item It is easy to show by induction on $N$ and since $\nu\lambda>1/e$ that $E_{\nu\lambda}^N(0)\geq E_{\nu\lambda}(N-1)$ and thus \begin{equation}\label{use}
					E_{\nu\lambda}(N-1)\leq c+\lambda\Rightarrow N \leq \log(c+\lambda)+2.
				\end{equation}
				
			\end{enumerate}

			By Lemma \ref{lipschitz} and since $\lambda> L^5$ we will now have that for all $x\in B(0,\delta)$ \begin{align*}|\mathcal{Z}_{\nu}^N(x)-\mathcal{Z}_{\nu}^N(0)|\leq&  E_{\nu\lambda}(\delta)\cdots E_{\nu\lambda}^N(\delta)|x|\\\leq&E^2_{\nu\lambda}(0)\cdots E_{\nu\lambda}^{N+1}(0)\delta\\\leq&(c+\lambda)^{N-1}E_{\nu\lambda}(\lambda+c)\delta\\\leq&(c+\lambda)^{\log(c+\lambda
					)+1}\nu\lambda e^{\lambda+c}\delta.\end{align*}
			
			Hence, by (\ref{key}) we will have that \begin{equation}\label{eq2}
				|\mathcal{Z}_{\nu}^N(x)-\mathcal{Z}_{\nu}^N(0)|\leq \lambda.
			\end{equation}

			Equation (\ref{eq2}) together with (ii) implies that $\mathcal{Z}_{\nu}^{N+2}(A_2\cup A_3)\subset B(\mathcal{Z}_{\nu}^{N}(0),\lambda)$ and by the choice of $c$ this last ball is contained in $\{(x_1,x_2,x_3)\in\mathbb{R}^3:x_3>c\}$. This implies that the part of $\mathcal{Z}_{\nu}^{N+2}(A_2\cup A_3)$ that lies below $S_0$ is contained in $A_1$.

			Recall that	 $\mathcal{Z}_\nu^n(V)$ stays below $S_0$ for $n>N_0$.
			By Lemma \ref{newlemma} now we know that there is a point $x_0\in \mathcal{Z}_{\nu}^{n_0}(V)$ for some $n_0>N_0$ such that  $x_0\in A_1$. Take such a point $x_0$. We now consider the iterates of this point. Let us examine the behaviour of those iterates more carefully. We can assume that $A_1$ is so close to the $x_3$-axis that if $y\in A_1$  then $p_3(\mathcal{Z}_{\nu}(y))>p_3(y)$ by Lemma \ref{upup}(a). Hence the points $\mathcal{Z}_{\nu}^n(x_0)$  go higher and higher up in the $x_3$ direction, while at the same time staying in $A_1$, until at some point the iterate $\mathcal{Z}_{\nu}^k(x_0)$, for some $k$, will lie in either $A_2$ or $A_3$ thanks to Lemma \ref{upup}(b).  Without loss of generality assume that $x_1:=\mathcal{Z}_{\nu}^{k}(x_0)\in A_2$  and take a small ball around $x_1$, $B(x_1,r)$ with $B(x_1,r)\subset A_2\cap \mathcal{Z}^{n_0+k}_{\nu}(V)$.  By what we have said in the previous paragraph now, we will have that $\mathcal{Z}_{\nu}^{N+2}(B(x_1,r))\subset A_1$.

			However, we know what happens in points inside $A_1$ when we iterate, they eventually leave $A_1$. Thus for some $k>N+2$ we will have that $\mathcal{Z}_{\nu}^{k}(B(x_1,r))\subset A_2 \cup A_3$ since $B(x_1,r)$ is a connected set and the sets $A_1$, $A_2$ and $A_3$ are disjoint. We can then repeat this whole argument, meaning take the set $\mathcal{Z}_{\nu}^{k}(B(x_1,r))$ which is now in $A_2$ or $A_3$ and thus will get mapped by $\mathcal{Z}_{\nu}$ to the lower half space $x_3<-\frac{\nu \lambda e^c}{2}$ and by $\mathcal{Z}_{\nu}^{N+2}$ inside $A_1$. Now continue as above and then repeat. Eventually we get a sequence $n_j\to\infty$ with $$\mathcal{Z}_{\nu}^{n_j}(B(x_1,r))\subset A_2 \cup A_3.$$ Then by using Lemma \ref{metro} we will have that $m(\mathcal{Z}_{\nu}^{n_j}(B(x_1,r)))\to \infty$ but that is impossible since $m(A_2\cup A_3)$ is finite.
		\end{proof}
		\section{Escaping set of the Zorich maps}
		In this section we prove that the escaping set is connected for those  Zorich maps, for which Theorem \ref{main} holds. Note that we  assume that $\lambda>L^5$ and $\nu>\sqrt{\frac{2L}{\lambda}}$. 
		
		The proof of this theorem closely follows Rempe's proof for the connectivity of the escaping set of the exponential family in \cite{REMPE2010}. Before we begin with the proof we need to define a few things. First, in this section, for simplicity and without loss of generality we will assume that our Zorich map sends $B_{(0,0)}$ in the half space $\{(x_1,x_2,x_3)\in\mathbb{R}^3:x_2\leq x_1\}$.  Let \[\mathbb{H}_0:=\{(x_1,x_2,x_3)\in\mathbb{R}^3:x_2<x_1\hspace{2mm}\text{and}\hspace{2mm}x_2>-x_1\}\] and similarly\[\mathbb{H}_1:=\{(x_1,x_2,x_3)\in\mathbb{R}^3:x_2>x_1\hspace{2mm}\text{and}\hspace{2mm}x_2>-x_1\}\] \[\mathbb{H}_2:=\{(x_1,x_2,x_3)\in\mathbb{R}^3:x_2>x_1\hspace{2mm}\text{and}\hspace{2mm}x_2<-x_1\}\]\[\mathbb{H}_3:=\{(x_1,x_2,x_3)\in\mathbb{R}^3:x_2<x_1\hspace{2mm}\text{and}\hspace{2mm}x_2<-x_1\}.\] Also, let  $T_0=T_{(0,0)}:=B_{(0,0)}\cap\mathbb{H}_0$ and \[T_{(i,j)}=T_{(0,0)}+i(\lambda,\lambda,0)+j(\lambda,-\lambda,0),\hspace{2mm}i,j\in\mathbb{Z}.\]
		
		Note that
		$T_1:=T_{(0,-1)}=B_{(0,-1)}\cap\mathbb{H}_1$, $T_2:=T_{(-1,-1)}=B_{(0,-1)}\cap\mathbb{H}_2$ and $T_3:=T_{(-1,0)}=B_{(0,0)}\cap\mathbb{H}_3$ and thus $T_{i}\subset \mathbb{H}_i,$ $i=0,1,2,3$. Define now $\Lambda_i :\mathbb{H}_i\to T_i$, $i=0,1,2,3$  to be the inverse branches of $\mathcal{Z}_{\nu}$ in $T_i$. We can extend those maps to $\overline{\mathbb{H}_i}\setminus\{0\}$ for $i=0,1,2,3$ and again those extended maps are injective. We will use the same symbols, $\Lambda_i$ to denote these extended maps.

		Now take $\gamma_0:=\{(0,0,x_3):x_3<0\}$ and inductively define \[\gamma_{k}:=\Lambda_0(\gamma_{k-1}),\] for all $k\geq 1$. Each of the sets $\gamma_{k}$, $k\geq 1$, is an injective curve inside $\overline{T_0}$. 
		
		We define now the set $\Gamma_0$  by \[\Gamma_0:=\bigcup_{k\geq0}\gamma_{k}.\]
		\begin{lemma}\label{con}
			If $U\subset \mathbb{R}^3$ is any open set with $U\cap \overline{\Gamma_0}\not=\emptyset$ then there is a $k_0\in\mathbb{N}$ with $\gamma_k\cap U\not=\emptyset$ for all $k>k_0$. In particular, $\bigcup_{k\geq k_0}\gamma_k$ is dense in $\overline{\Gamma_0}$.
		\end{lemma}
		\begin{proof}
			Let $x_0\in\overline{\Gamma_0}$, and $U$ a neighbourhood of this point. We want to show that $\gamma_{k}\cap U\not=\emptyset$ for all sufficiently large $k$. We know, from the definition of $\overline{\Gamma_0}$, that there is a point $x_1\in U\cap\Gamma_0 $. This implies that $\mathcal{Z}_{\nu}^n(x_1)$ belongs to the $x_3$-axis for all  $n\geq N_0$, for some $N_0\in\mathbb{N}$ and in fact we can assume that $x_2:=\mathcal{Z}_{\nu}^{N_0}(x_1)\in H_{>M},$ where $M$ is any positive number. Now taking $M>M_0$, where $M_0$ is the constant we used in Lemma \ref{lemma0} and  applying that lemma for a ball $B(x_2,R)\subset\mathcal{Z}_{\nu}^{N_0}(U)$ $n$ times we get \[\mathcal{Z}_{\nu}^{n}\left(B(x_2,R)\cap H_{>M}\right)\supset B\left(\mathcal{Z}^n_{\nu}(x_2),\alpha^{-n}R\right)\cap H_{>E^n_{\nu\lambda}(M)}.\]
			For all large enough $n$ now the ball on the right hand side, $B\left(\mathcal{Z}^n_{\nu}(x_2),\alpha^{-n}R\right)$, intersects the line $\gamma_1=\{(2\lambda,0,t):t\in\mathbb{R}\}$. Hence, for all large enough $n$, $\mathcal{Z}_{\nu}^n(U)$ intersects $\gamma_1$. Thus for each $n$ large enough there is a point $x_3\in\gamma_1$ whose backward orbit intersects $U$ itself. This means that  $U$ contains a point in $\gamma_{k}$ for all large enough $k$ as we wanted.
		\end{proof}
		
		\begin{lemma}
			The set $\Gamma_0$ is connected.
		\end{lemma}
		\begin{proof}
			Suppose $U\subset\mathbb{R}^3$ is an open set with $U\cap\Gamma_0\not=\emptyset$ and $\Gamma_0\cap\partial U=\emptyset$. We show that $\Gamma_0\subset U$.
			
			By Lemma \ref{con} we have that $\gamma_k\cap U\not=\emptyset$, for all $k>k_0$. Since $\gamma_k$ is a connected curve this implies that $\gamma_k\subset U$, for all $k> k_0$. Thus \[\Gamma_0\subset \overline{\Gamma_0}= \overline{\bigcup_{k\geq k_0}\gamma_k}\subset \overline{U}.\]
			Hence, since $\Gamma_0\cap\partial U=\emptyset$, we have that $\Gamma_0\subset U$.
		\end{proof}
		Similarly now we can define sets $\Gamma_i$, for $i=1,2,3$ using this time $\Lambda_i$ instead of $\Lambda_0$ and prove that $\Gamma_i$ is also connected. This implies that the union $\Gamma:=\bigcup_{i=0}^3\Gamma_i$ is a connected set. We define now the set\[Y:=\bigcup_{(k,l)\in\mathbb{Z}^2}\left(\Gamma+k(2\lambda,2\lambda,0)+l(2\lambda,-2\lambda,0)\right),\] which is  connected since $\Gamma$ contains the lines $\{(\pm2\lambda,0,t):t\in\mathbb{R}\}$, $\{(0,\pm 2\lambda,t):t\in\mathbb{R}\}$ and is a subset of $I(\mathcal{Z}_{\nu})$ since the iterates of any point eventually land on the $x_3$-axis. Next we define, inductively, the sets $Y_j\subset I(\mathcal{Z}_{\nu})$ by setting $Y_0=Y$ and $Y_{j+1}=\mathcal{Z}_{\nu}^{-1}(Y_j)\cup Y_j$.
		
		\begin{lemma}
			The sets $Y_j$ are connected for all $j\geq0$.
		\end{lemma}
		\begin{proof}
			We will prove this by induction on $j$. Let us define the inverse branches of $\mathcal{Z}_{\nu}$. By using the notation we introduced in the first paragraphs of this section  define $\Lambda_{k,l}:\mathbb{H}_p\to T_{(k,l)}$, with $p=0,1,2,3$ to be the inverse branches of $\mathcal{Z}_{\nu}$ that take values on the square beams $T_{(k,l)}$. We also extend those maps to $\overline{\mathbb{H}_p}\setminus\{0\}$ and use the same symbol to denote those extensions. With that notation we have that \[Y_{j+1}=\bigcup_{(k,l)\in\mathbb{Z}^2}\Lambda_{k,l}(Y_j)\cup Y_j.\]  By the inductive hypothesis now we know that $Y_j$ is connected and because $\Lambda_{k,l}$ is continuous the set  $\Lambda_{k,l}(Y_j)$ is also connected. Observe now that the point $x_{n,m}=(2\lambda,0,0)+n(2\lambda,0,0)+m(0,2\lambda,0)$ is inside $Y=Y_0$  and thus in $Y_j$, for all $n,m\in \mathbb{Z}$ and for all $j\in\mathbb{N}$. Also note that $\mathcal{Z}_{\nu}(x_{n,m})=(0,0,-\nu\lambda)$ or $(0,0,\nu\lambda)$  which are both points in $Y_j$. This means that there are $m, n$ depending on $k,l$ such that $x_{n,m}\in\Lambda_{k,l}(Y_j)$. Hence $\Lambda_{k,l}(Y_j)\cap Y_j\not=\emptyset$. This implies that the set $\Lambda_{k,l}(Y_j)\cup Y_j$ is connected. Hence $Y_{j+1}$ is connected as a union of connected sets with non-empty intersections with each other as  we wanted.
		\end{proof}
		\begin{proof}[Proof of Theorem \ref{escaping}]
			Consider the set \[\bigcup_{j\geq 0}\mathcal{Z}_{\nu}^{-j}((0,0,-1))\subset\bigcup_{j\geq 0}Y_j .\] The Zorich map is bounded on $\{(x_1,x_2,x_3):x_3<0\}$ and thus it does not have the pits effect (see \cite{B-Nicks}).  Hence, by \cite[Theorem 1.8]{B-Nicks} we will have that the set $\bigcup_{j\geq 0}\mathcal{Z}_{\nu}^{-j}((0,0,-1))$ is dense in $\mathcal{J}(\mathcal{Z}_\nu)$, which by Theorem \ref{main} is $\mathbb{R}^3$, and thus also dense in $I(\mathcal{Z}_{\nu})$. Thus the set $\bigcup_{j\geq 0}Y_j $ is a connected dense subset of $I(\mathcal{Z}_{\nu})$ which implies that the escaping set itself is connected.
		\end{proof}
		
		\section{Density of periodic points}
		\begin{proof}[Proof of Theorem \ref{density}]
			First let $U_0=B(x_0,r)$ be a ball centred at  $x_0\in\mathbb{R}^3$ of radius $r>0$. We seek a periodic point of $\mathcal{Z}_\nu$ in $U_0$. Without loss of generality we may assume that $\overline{U_0}$ does not intersect any of the planes $x_1=\pm x_2+2\lambda k$, $k\in\mathbb{Z}$. 
			
			We will follow the method of \cite{Fletcher2013} where the authors prove that periodic points of a quasiregular version of the sine function are dense on $\mathbb{R}^3$. We will do this by finding an $N\in\mathbb{N}$ and a finite sequence of open sets $U_j$, $j=0,\dots, N$ such that 
			\begin{enumerate}[label=(\roman*)]
				\item $U_{j+1}\subset \mathcal{Z_{\nu}}(U_j)$, $0\leq j\leq N-1$.
				\item $\mathcal{Z_{\nu}}$ is a homeomorphism on each $U_j$ for $j\leq N-1$.
				\item $\overline{U_0}\subset U_N$.
			\end{enumerate} 
			If these conditions are met then we can define a continuous inverse branch  $\mathcal{Z}^{-N}_{\nu}:U_N\to U_0$. Thus by the Brouwer fixed point theorem the map ${\mathcal{Z}^{-N}_{\nu}}|_{U_0}$ has a fixed point in $U_0$.
			
			We will now show how we can construct such a sequence. By Theorem \ref{main} we know that $\mathcal{Z}^n_\nu(U_0)$ eventually covers $\mathbb{R}^3\setminus\{0\}$. We set $U_j=\mathcal{Z}^j_\nu(U_0)$  for all $j$ such that $\mathcal{Z}^j_\nu(U_0)$ does not intersect the set $P:=\bigcup_{k\in\mathbb{Z}}\{(x_1,x_2,x_3):x_1=\pm x_2+2k\lambda\}$. Let $n_0$ be the  biggest such $j$, so that we have defined $U_0,\dots, U_{n_0}$. Then take a point $y_1$ in $\mathcal{Z}^{n_0+1}(U_0)\cap P$ such that $y_1\not\in B_{\mathcal{Z}_\nu}$ and a ball $B(y_1,r)\subset \mathcal{Z}^{n_0+1}_\nu(U_0)\setminus B_{\mathcal{Z}_\nu}$, where we remind here that $B_{\mathcal{Z}_\nu}$ is the branch set. Set $U_{n_0+1}=B(y_1,r)$. We know that $\mathcal{Z}_\nu(U_{n_0+1})$ intersects one of the planes $x_1=\pm x_2$ and it is easy to see that $\mathcal{Z}_\nu$ is a homeomorphism on $U_{n_0+1}$. Assume, without loss of generality that it intersects $x_1=x_2$ and take $y_2\in( \mathcal{Z}_\nu(U_{n_0+1})\cap\{(x_1,x_2,x_3):x_1=x_2\})\setminus B_{\mathcal{Z}_\nu}$. Set $U_{n_0+2}=B(y_2,r_2)$, where $r_2>0$ is such that $B(y_2,r_2)\subset \mathcal{Z}_\nu(U_{n_0+1})\setminus B_{\mathcal{Z}_\nu}$.
			
			Consider the set $V_0=U_{n_0+2}\cap\{(x_1,x_2,x_3):x_1=x_2\}$ which is an open set of the plane $x_1=x_2$ in the subspace topology. We define the sets $V_n$ by induction as follows. Suppose that $V_n$ has been defined and that $V_n\cap B_{\mathcal{Z}_\nu}=\emptyset$. We consider now two cases:
			\begin{enumerate}
				\item $\mathcal{Z}_\nu(V_n)$  intersects one of the lines $\{(x_1,x_2,x_3):x_1=x_2=2\lambda k\}$, $k\in\mathbb{Z}$ which are the pre-images of the $x_3$ axis on the plane $x_1=x_2$.
				\item $\mathcal{Z}_\nu(V_n)$ does not intersect any of those lines.
			\end{enumerate} 
			In the first case, let $y_3$ be a point in such an intersection. We define $V_{n+1}$ to be an open ball around $y_3$ in the subspace topology of $x_1=x_2$ of radius  $r_3$ where $r_3$ is chosen in such a way that the ball does not contain branch points and such that $V_{n+1}\subset \mathcal{Z}_{\nu}(V_n)$.
			
			In the second case, we define $V_{n+1}:=\mathcal{Z}_\nu(V_n)\cap H_0$, where $H_0$ is the whole plane $x_1=x_2$ in case $\mathcal{Z}_\nu(V_n)\cap B_{\mathcal{Z}_\nu}=\emptyset$ and it is an open half plane on $x_1=x_2$, in any other case, which is defined as follows. Suppose that $\mathcal{Z}_\nu(V_n)$ intersects one of the lines of $B_{\mathcal{Z}_\nu}$ which we call $\ell_1$. We set $H_0$ to be the half plane defined by this line and the property \[m_2(\mathcal{Z}_\nu(V_n)\cap H_0)\geq \frac{1}{2}m_2(\mathcal{Z}_\nu(V_n)),\] where $m_2$ is the 2 dimensional Lebesgue measure on $x_1=x_2$. Note now that we have inductively defined the sets $V_n$. 
			
			We now claim that case (1) must occur for some $n$, otherwise notice that by construction $\mathcal{Z}_\nu$ is a homeomorphism on $V_n$, for all $n\in\mathbb{N}$. Hence, \begin{equation}\label{nixta}m_2(V_{n+1})=m_2(\mathcal{Z}_\nu(V_n)\cap H_0)\geq\frac{1}{2}m_2(\mathcal{Z}_\nu(V_n)).\end{equation} 
			Using the notation of section \ref{section3} we now have that \[m_2(\mathcal{Z}_\nu(V_n))=m_2((\phi^{-1}\circ g\circ \phi)(V_n))=Cm_2(g(\phi(V_n))), \]
			where  $C=|\det D\phi^{-1}|$ which is a constant since $\phi$ is linear.
			Combining with equation \eqref{nixta} this gives $$m_2(V_{n+1})\geq\frac{C}{2} m_2(g(\phi(V_n))).$$
			Thus by Lemma \ref{volume} we have that $m_2(V_n)\to \infty$, as $n\to \infty$. This implies, just like in the proof of Theorem \ref{plane}, that there is an $m_0$ such that $V_{m_0}$ intersects the $x_3$ axis.
			
			We now set $U_{n_0+2+i}$ to be the open set \[U_{n_0+2+i}:=\bigcup_{x\in V_i}B(x,r_x),\]

			where  $r_x\hspace{1mm}\text{is such that}\hspace{1mm} B(x,r_x)\subset\mathcal{Z}_\nu(U_{n_0+1+i})\hspace{1mm}\text{and}\hspace{1mm} B(x,r_x)\cap B_{\mathcal{Z}_\nu}=\emptyset$, for all $1~\leq~i~\leq~m_0$.
			
			Notice that the sets $U_{n_0+2+i}$ satisfy the properties (i) and (ii). We have that $U_{n_0+2+m_0}$ intersects the $x_3$ axis, so let $$U_{n_0+3+m_0}= \mathcal{Z}_\nu(U_{n_0+2+m_0})\cap B_{(0,0)}.$$
			
			Define also \begin{equation*}Q_{(0,0)}=\{(x_1,x_2):|x_1|+|x_2|<2\lambda \},\end{equation*} and $$Q_{(k,l)}=Q_{(0,0)}+k(2\lambda,2\lambda)+l(2\lambda,-2\lambda),\hspace{2mm}k,l\in\mathbb{Z}.$$
			
			We now set $$U_{n_0+2+m_0+j}=\mathcal{Z}_\nu(U_{n_0+1+m_0+j})\cap B_{(0,0)} ,$$ for all $2 \leq j\leq m_1$, where $m_1$, depending on $M>0$, is so large that $U_{n_0+2+m_0+m_1}$ contains a set of the form $Q_0\times [R,R+M]$, where $Q_0=Q_{(0,0)}\cap \{(x_1,x_2):x_1<x_2\}$ and  some $R>0$. We know that such an $m_1$ exists because the iterated image, under the Zorich map, of an open set that intersects the $x_3$-axis eventually contains a ball of radius as large as we want (see Lemma \ref{lemma0}).   
			
			If $M$ is large enough then $\mathcal{Z}_\nu(U_{n_0+2+m_0+m_1})$ will contain a set of the form \[U_{n_0+3+m_0+m_1}:=Q_{(k,l)}\times[-t_M,t_M],\]  for some $k,l\in\mathbb{Z}$ and  $t_M\to\infty$, as $M\to \infty$. Note that $\mathcal{Z}_\nu$ is a homeomorphism on $U_{n_0+2+m_0+j}$ for all $2\leq j\leq m_1+1$ and that $\mathcal{Z}_\nu(U_{n_0+3+m_0+m_1})$ will be the set \[U_N:=\{x\in\mathbb{R}^3:\nu\lambda e^{-t_M}< |x|< \nu\lambda e^{t_M}\}\setminus W ,\]
			where $W=\{(x_1,x_2,x_3):x_1=\pm x_2, \hspace{2mm}x_3\leq 0\}$.
			If $M$ is large enough $U_N$ will contain the closure of our initial set $U_0$, since $U_0$ does not intersect any of the planes $x_1=\pm x_2 +2\lambda k$, $k\in\mathbb{Z}$ and we are done. 
		\end{proof}
		\section{Generalized Zorich Maps}\label{pyramida}
		
		In this section we discuss a more general construction of Zorich maps. 
		The goal of this section is to sketch how to prove Theorem \ref{pyramid} by following the same methods we used in the proof of Theorem \ref{main} and highlight the most significant differences between the two cases.\\
		We start again with the square 
		\[Q=\Big \{(x_1,x_2)\in\mathbb{R}^2:|x_1|\leq 1,|x_2|\leq 1\Big \}.\] The $L$ bi-Lipschitz function $\mathfrak{h}_{gen}:Q\to\mathbb{R}^3$ maps this square to a surface $\mathcal{S}$ which satisfies the following:

		\begin{enumerate}
			\item The surface lies in the half space $\{(x_1,x_2,x_3):x_3\geq 0\}$.
			\item The boundary of $\mathcal{S}$ lies on the plane $x_3=0$.
			\item The ray that connects $(0,0,0)$ with $\mathfrak{h}_{gen}(x)$, $x\in Q$, intersects the surface $\mathcal{S}$ only at $\mathfrak{h}_{gen}(x)$.
			\item\label{itemangle} There is a $\theta_{_\mathcal{S}}\in(0,\pi/2)$ and $\varepsilon>0$ such that for all  points $w,z \in \mathcal{S}$ such that $|w-z|\leq \varepsilon$ the acute angle between the lines connecting $0$ with $z$ and $w$ with $z$ is greater than $\theta_{_\mathcal{S}}$. We will call this property the \textit{non-tangential position vector} property.
			
			\item $\min_{x\in Q}|\mathfrak{h}_{gen}(x)|>0$.
		\end{enumerate}

		\begin{remark2}
			We make two observations on the  non-tangential position vector property that we are going to need later.
			
			First we note that it implies that for all points $x\in\mathcal{S}$ for which a tangent plane to $\mathcal{S}$ is defined at $x$ (we know that this includes Lebesgue almost all points of $\mathcal{S}$) the angle between the vector $\mathfrak{h}_{gen}(x)$ and the plane is at least $\theta_{_\mathcal{S}}$. 
			
			Second, consider any straight line segment inside $Q$, which we can parametrize by $\phi(t),$ $t\in[0,1]$ and $\phi$ linear,  and consider $h(\phi([0,1]))$ which is a curve in $\mathcal{S}$ that admits a tangent line almost everywhere. The non-tangential position vector property now implies that the angle between the vector $h(\phi(t))$ and the tangent line at that point on the surface is again at least $\theta_{_\mathcal{S}}$. 
			
			We also note here that a similar condition to the   non-tangential position vector property was used by Nicks and Sixsmith in \cite{N-S} on the boundary of a domain in $\mathbb{R}^d$ in order to prove an extension theorem on bi-Lipschitz maps between domains.
		\end{remark2}	
		Again if $\mathfrak{h}_{gen}=(\mathfrak{h}_{gen,1},\mathfrak{h}_{gen,2},\mathfrak{h}_{gen,3})$ we require that $\mathfrak{h}_{gen,1}(x_1,x_1)=\mathfrak{h}_{gen,2}(x_1,x_1)$ and $\mathfrak{h}_{gen,1}(x_1,-x_1)=-\mathfrak{h}_{gen,2}(x_1,-x_1)$. For simplicity we will also assume that 
		\begin{equation}\label{eqgen}
			\mathfrak{h}_{gen}(0,0)=(0,0,1)\hspace{2mm}\text{and that }\hspace{2mm}
			\sup_{x\in Q}|\mathfrak{h}_{gen}(x)|~=~1.
		\end{equation} 
		Although the last two conditions are not needed for our methods to work, they make the arguments less arduous and more similar with the arguments we used in the more classical setting.

		We also rescale our map $\mathfrak{h}_{gen}$ by defining\[h_{gen}(x_1,x_2)=\lambda \mathfrak{h}_{gen}\left(\frac{1}{\lambda}(x_1,x_2)\right), \hspace{2mm}(x_1,x_2)\in \lambda Q.\]

		We then define \[\mathcal{Z}_{gen}(x_1,x_2,x_3)=e^{x_3}h_{gen}(x_1,x_2),\] on $\lambda Q\times \mathbb{R}$ and extend this map to the whole $\mathbb{R}^3$ through reflections.

		\begin{remark2}
			
			A case of particular interest is when the surface $\mathcal{S}$ is a square based pyramid. In this case we can be much more explicit and define the function $h_{pyr}:\lambda Q\to\mathbb{R}^3$, \[h_{pyr}(x_1,x_2):=\left(x_1,x_2,\lambda-\max\{|x_1|,|x_2|\}\right)\] which sends the square $\lambda Q$ to a pyramid with base $\lambda Q$ and height $\lambda$. We then define on $\lambda Q\times\mathbb{R}$ $$\mathcal{Z}_{pyr}(x_1,x_2,x_3)=e^{x_3}h_{pyr}(x_1,x_2)$$ and extend this map to all $\mathbb{R}^3$ in the same way we did with the classical Zorich map. 
			
			Also let us mention that in \cite{N-S} the authors used those kind of Zorich maps to construct  a quasiregular function in $\mathbb{R}^3$ which resembles $e^z+z$.
			
			For those maps we can prove (although we omit the proof), using the same methods,  the corresponding result to Theorem \ref{main} where we have a more explicit value for the scale factor.
		\end{remark2}
			\begin{theorem}
				For $\lambda> 2$ the Julia set $\mathcal{J}(\mathcal{Z}_{pyr})$  is the entire $\mathbb{R}^3$.
			\end{theorem}

		We are ready now to discuss the proof of Theorem \ref{pyramid}.
		
		First we have to show that the $x_3$-axis belongs to the Julia set which is proven in exactly the same way as for the spherical Zorich maps (see Proposition \ref{prop}) so we omit the proof. Then we have to study our maps in the planes $x_1=\pm x_2$. Again in those planes our map is conjugate through $\phi(x_1,x_2,x_3)=\frac{1}{\lambda}(x_3+ix_1)$ to the map 
		
		\[\hat{g}(z):= \begin{cases}
			\hat{\psi}\left(\bar{z}+2 i\right),&\hspace{1mm}\imag (z)\in\left[(4k+1),(4k+3)\right]\\\\\hat{\psi}(z)&\hspace{1mm}\imag (z)\in\left[(4k-1),(4k+1)\right]
		\end{cases},\]where $z=x+iy\in\mathbb{C}$, $k\in\mathbb{Z}$ and $\hat{\psi}(x+iy)= e^{\lambda x}\left(\mathfrak{h}_{gen,3}(y,y)+i\mathfrak{h}_{gen,1}(y,y)\right)$.
		We again set $\mathfrak{a}(y)=\mathfrak{h}_{gen,3}(y,y)$ and $\mathfrak{b}(y)=\mathfrak{h}_{gen,1}(y,y)$.
		
		We can then prove  that Theorem \ref{plane} holds in this setting as well. 
		\begin{lemma}
			For $\lambda> \frac{2L^2}{\sin \theta_{_{\mathcal{S}}}\min_{x\in Q}|\mathfrak{h}_{gen}(x)|}$ if $V$ is a connected  set of the complex plane with $m(V)>0$ then $\hat{g}^n(V)$ intersects the real axis for some $n\in \mathbb{N}$.
		\end{lemma}
		This of course implies that the planes $x_1=\pm x_2$ and all their parallel translate planes $x_1=\pm x_2+ 2\lambda k$, $k\in\mathbb{Z}$ are in $\mathcal{J}(\mathcal{Z}_{gen})$. Again all those planes partition $\mathbb{R}^3$ in square beams whose boundaries are in the Julia set and in which our Zorich map is a homemorphism.
		
		We will not give the proof of the above lemma here since it is very similar with the proof of Theorem \ref{plane}. The only significant difference in the proof of the above lemma in this more general setting is in the corresponding Lemma \ref{expand} which we prove below.

		\begin{lemma}
			\[|\det (D\hat{g}(z))|\geq \frac{\sin \theta_{_{\mathcal{S}}}\min_{x\in Q} |\mathfrak{h}_{gen}(x)|\lambda e^{2\lambda\real(z)}}{2L}\hspace{1mm}\text{a.e.}\]
		\end{lemma}
		\begin{proof}
			The only difference with the proof of Lemma \ref{expand} is in finding a lower bound for \[\left|\det\begin{pmatrix} \mathfrak{a}(y)& \frac{d\mathfrak{a}}{dy}({y})\\\\
				\mathfrak{b}(y)& \frac{d\mathfrak{b}}{dy}(y)\end{pmatrix}\right|.\]
			
			This time we know that the absolute value of the determinant equals \[\left|\left(\mathfrak{a}(y),\mathfrak{b}(y)\right)\right|\left|\left(\frac{d\mathfrak{a}}{dy}({y}),\frac{d\mathfrak{b}}{dy}({y})\right)\right||\sin \theta(y)|,\]
			
			where $\theta(y)$ is the angle between the vectors $(\mathfrak{a}(y),\mathfrak{b}(y))$ and $(\frac{d\mathfrak{a}}{dy}({y}),\frac{d\mathfrak{b}}{dy}(y))$.

			Hence using the non-tangential position vector property and the fact that $$\left|\left(\frac{d\mathfrak{a}}{dy}({y}),\frac{d\mathfrak{b}}{dy}({y})\right)\right|\geq \frac{1}{\sqrt2 L}\hspace{2mm}\text{and}\hspace{2mm}\left|\left(\mathfrak{a}(y),\mathfrak{b}(y)\right)\right|\geq \frac{\min_{x\in Q} |\mathfrak{h}_{gen}(x)|}{\sqrt2}$$ we get that 
			
			\[\left|\det\begin{pmatrix} \mathfrak{a}(y)& \frac{d\mathfrak{a}}{dy}({y})\\\\
				\mathfrak{b}(y)& \frac{d\mathfrak{b}}{dy}(y)\end{pmatrix}\right|\geq \frac{\sin \theta_{_{\mathcal{S}}}\min_{x\in Q} |\mathfrak{h}_{gen}(x)|}{2L}\]
			and thus
			\[|\det D\hat{g}(z)|\geq\frac{ \sin \theta_{_\mathcal{S}}\min_{x\in Q} |\mathfrak{h}_{gen}(x)|\lambda e^{2\lambda \real{z}}}{2 L}\hspace{2mm}\text{a.e.}\]
		\end{proof}

		Next we need the Misiurewicz type Lemma \ref{dete} which in this case becomes 
		\begin{lemma}\label{detev2}
			\[\det\left(D\mathcal{Z}^n_{gen}(x)\right)\geq \left(\frac{\lambda\min_{x\in Q}|\mathfrak{h}_{gen}(x)|\sin \theta_{_{\mathcal{S}}}}{L^5}\right)^n\frac{1}{\lambda^3}\left|\left( p \circ \mathcal{Z}_{gen}^{n}\right)(x)\right|^3\hspace{2mm}\text{a.e.}\]
		\end{lemma}
		\begin{proof}
			Again the only difference is in obtaining a lower bound for the determinant $\det D\mathcal{Z}_{gen}(x)$. Define $$\mathcal{H}=\begin{pmatrix}
				\frac{\partial h_{gen,1}}{\partial x_1}(p(x))&\frac{\partial h_{gen,1}}{\partial x_2}(p(x))&h_{gen,1}(p(x))\\\\\frac{\partial h_{gen,2}}{\partial x_1}(p(x))&\frac{\partial h_{gen,2}}{\partial x_2}(p(x))&h_{gen,2}(p(x))\\\\\frac{\partial h_{gen,3}}{\partial x_1}(p(x))&\frac{\partial h_{gen,3}}{\partial x_2}(p(x))&h_{gen,3}(p(x))
			\end{pmatrix}$$

			and set \[\mathcal{A}=\begin{pmatrix}
				\frac{\partial h_{gen,1}}{\partial x_1}\left(p(x)\right)\\\\\frac{\partial h_{gen,2}}{\partial x_1}\left(p(x)\right)\\\\\frac{\partial h_{gen,3}}{\partial x_1}\left(p(x)\right)
			\end{pmatrix},\hspace{2mm} \mathcal{B}=\begin{pmatrix}
				\frac{\partial h_{gen,1}}{\partial x_2}\left(p(x)\right)\\\\\frac{\partial h_{gen,2}}{\partial x_2}\left(p(x)\right)\\\\\frac{\partial h_{gen,3}}{\partial x_2}\left(p(x)\right)\end{pmatrix},\hspace{2mm} \mathcal{C}=\begin{pmatrix}
				h_{gen,1}\left(p(x)\right)\\\\h_{gen,2}\left(p(x)\right)\\\\h_{gen,3}\left(p(x)\right)
			\end{pmatrix}.
			\]
			
			Then $\det\mathcal{H}=\big\langle \mathcal{A}\times \mathcal{B}, \mathcal{C}\big\rangle=|\mathcal{A}\times\mathcal{B}||\mathcal{C}| \cos \phi,$ where $\phi$ is the angle between $\mathcal{A}\times \mathcal{B}$ and $\mathcal{C}$. Using  now the fact that $|\mathcal{C}|\geq\lambda \min_{x\in Q}|\mathfrak{h}_{gen}(x)|$ and  $|\mathcal{A}\times\mathcal{B}|\geq \frac{1}{L^2}$ together with  the non-tangential position vector property we can show that \[\det\mathcal{H}\geq \frac{\lambda \min_{x\in Q}|\mathfrak{h}_{gen}(x)|\sin \theta_{_\mathcal{S}} }{L^2}.\]
			
			Hence \[\det D\mathcal{Z}_{gen}(x)\geq e^{3x_3}\frac{\lambda \min_{x\in Q}|\mathfrak{h}_{gen}(x)|\sin \theta_{_\mathcal{S}}}{L^2}\]
			and the rest follows in exactly the same way as in the proof of Lemma \ref{dete}.
		\end{proof}
		Versions of Lemmas \ref{upup}, \ref{metro}, \ref{surfacelemma}, \ref{axis}, \ref{newlemma}, \ref{lipschitz} now follow with only slight modifications on their proofs. Hence, the proof of Theorem \ref{pyramid} now follows with the same arguments as the proof of Theorem \ref{main}. Let us briefly sketch how all this should work.
		\begin{lemma}\label{upupv2}
			\begin{enumerate}
				
				\item[(a)] There are $\delta>0 $ and  $c>0$ such that for all $x\in C_{\delta}$, where $C_{\delta}$ is the cylinder around $x_3$-axis with $\delta$ radius, we have that $p_3(\mathcal{Z}_{gen}(x))>p_3(x)+c$.
				\item[(b)] For  $\delta$ as in (a) and for every $x\in C_{\delta}$, with $p(x)\not=(0,0)$, there is an $n\in\mathbb{N}$ such that $\mathcal{Z}_{gen}^n(x)\not\in C_{\delta}$.
			\end{enumerate}
			
		\end{lemma}
		\begin{proof}
			The proof of (a) goes word for word as Lemma \ref{upup}. For (b) again the proof is almost the same. The difference here is the lower bound for the angle $\theta$ used in the proof of Lemma \ref{upup} where instead of $\frac{\pi}{4}$ is now some constant larger than $0$.
		\end{proof}
		\begin{lemma}\label{metrov2}
			Assume $\lambda>\frac{L^5}{\min_{x\in Q}|\mathfrak{h}_{gen}(x)|\sin \theta_{_{\mathcal{S}}}}$.	Let $V\subset\mathbb{R}^3$ be a connected set with $m(V)>0$ and whose iterates do  not intersect any of the planes $x_1=\pm x_2+2k\lambda$, $k\in\mathbb{Z}$ . Suppose also that there is a sequence  of integers $n_j>0$ with $\mathcal{Z}_{gen}^{n_j}(V)\cap C_a=\emptyset$, where $C_a$ is a cylinder around $x_3$-axis of any radius $a>0$. Then $m(\mathcal{Z}_{gen}^{n_j}(V))\to\infty$ as $n_j\to \infty$, where $m$ is the 3-dimensional Lebesgue measure.
		\end{lemma}
		\begin{proof}
			The proof is the same as in Lemma \ref{metro} only now we use Lemma \ref{detev2} in place of Lemma~\ref{dete}.
		\end{proof}
		In the same way, as for the Zorich map defined using spheres, we can define the surfaces $S_n$ and $K_n$ lying inside the rectangle beams $B_{(0,0)}$ and $B_{(0,-1)}$ respectively. Again those surfaces, together with the plane $x_3=0$ and the boundaries of the beams,  define sets of finite volume. 
		
		The next three lemmas are the corresponding ones to Lemmas \ref{surfacelemma}, \ref{axis}, \ref{newlemma} respectively. Their proofs almost go word for word with the proofs of the lemmas we just mentioned and are therefore omitted.
		\begin{lemma}\label{surfacelemmav2}
			Let $I_n$ be the volume that the surface $S_n$ encloses together with the plane  $x_3=0$ and inside the beam $B_{(0,0)}$. Then $I_n$ is finite for all $n\in\mathbb{N}$. Furthermore, if $T_n:=I_{n+1}-I_n$ is the volume between $S_n$ and $S_{n+1}$ then $T_n$ is a decreasing sequence.  
		\end{lemma}
		
		\begin{lemma}\label{axisv2}
			Assume $\lambda>\frac{L^5}{\min_{x\in Q}|\mathfrak{h}_{gen}(x)|\sin \theta_{_{\mathcal{S}}}}$. Let $V$ be a connected subset of $\mathbb{R}^3$ with $m(V)>0$ and such that $\mathcal{Z}_{gen}^n(V)$ does not intersect any of the planes $x_1=\pm x_2+2k\lambda$, $k\in\mathbb{Z}$ for all $n\in\mathbb{N}$. Then $\mathcal{Z}_{gen}^n(V)$ visits infinitely often one of the two rectangle beams $B_{(0,0)}$, $B_{(0,-1)}$, that have the $x_3$-axis in their boundary.
		\end{lemma}
		
		\begin{lemma}\label{newlemmav2}
			Assume $\lambda>\frac{L^5}{\min_{x\in Q}|\mathfrak{h}_{gen}(x)|\sin \theta_{_{\mathcal{S}}}}$. Let $V$ be a connected set of $\mathbb{R}^3$ with $m(V)>0$ and such that $\mathcal{Z}_{gen}^n(V)$ does not intersect any of the planes $x_1=\pm x_2+2k\lambda$, $k\in\mathbb{Z}$	for all $n\in\mathbb{N}$. Suppose that there is an $N_0\in \mathbb{N}$ such that  $\mathcal{Z}_{gen}^n(V)\subset B_{(0,0)}$, for all $n>N_0$. Then for all $M>0$ and $\varepsilon>0$ there is some $n_0>N_0$ and a point $x\in \mathcal{Z}_{gen}^{n_0}(V)$ such that $p_3(x)>M$ and $d(x,x_3\text{-axis})<\varepsilon$, where $d$ is the euclidean distance.
		\end{lemma}
		
		The next Lemma is the analogue of Lemma \ref{lipschitz} in this new setting.
		\begin{lemma}\label{lipschitzv2}
			Let $y_1,y_2\in B(0,r)$, where $r>0$. Then  for all $n\in \mathbb{N}$ it is true that \[|\mathcal{Z}_{gen}^n(y_1)-\mathcal{Z}_{gen}^n(y_2)|\leq \left(\frac{\sqrt{L^2+\lambda^2}}{\lambda}\right)^nE_{\lambda}(r)\cdots E_{\lambda}^n(r)|y_1-y_2|,\] where $E_{\lambda}$ denotes the exponential map $\lambda e^x$.
		\end{lemma}
		\begin{proof}
			The proof almost goes word for word with the proof of Lemma \ref{lipschitz}. Note however that $\mathcal{A}$, $\mathcal{B}$,  $\mathcal{C}$ are not orthogonal. Still when estimating $|D\mathcal{Z}_{gen}(x_1,x_2,0)|$ (see proof of Lemma \ref{lipschitz}) we can argue as follows 
			\begin{align*}
				|D\mathcal{Z}_{gen}(x_1,x_2,0)|^2&=\sup_{|v|=1}\left(|v_1\mathcal{A}+v_2\mathcal{B}+v_3\mathcal{C}|^2\right)\\&\leq\sup_{|v|=1}\left(|v_1\mathcal{A}+v_2\mathcal{B}|+|v_3\mathcal{C}|\right)^2\\&\leq\sup_{|v|=1}\left(L|(v_1,v_2)|+\lambda|v_3|\right)^2\\&\leq L^2+\lambda^2
			\end{align*}
			We also note that to argue here as in the last few lines of the proof of Lemma \ref{lipschitz} we use the two conditions in equation \eqref{eqgen}.
		\end{proof}
		\begin{proof}[Proof of Theorem \ref{pyramid}]
			Let $V$ be any open and connected set of $\mathbb{R}^3$. Assuming that $$\lambda>C_{h_{gen}}:=\frac{\max \{L^5,2 L\}}{\min_{x\in Q}|\mathfrak{h}_{gen}(x)|\sin \theta_{_{\mathcal{S}}}}$$ we want to show that $\mathcal{Z}^n_{gen}(V)$ intersects one of the planes that belong to the Julia set for some $n$ and thus $V$ itself intersects the Julia set.
			
			The proof now proceeds in the same way as the proof of Theorem \ref{main}. We consider the same two cases: \begin{enumerate}[label=(\roman*)]
				\item The iterates $\mathcal{Z}^n_{gen}(V)$ do not eventually stay inside the beam $B_{(0,0)}\cup B_{(0,-1)}$. In this case the proof is the same almost word for word.
				\item The iterates $\mathcal{Z}^n_{gen}(V)$ eventually stay inside $B_{(0,0)}\cup B_{(0,-1)}$. The idea in this case will be the same. We leave the details, which will be slightly different, to the interested reader.
			\end{enumerate}
		\end{proof}
		
		\section{Questions and Remarks}\label{remarks}
		As we have already seen the Zorich maps resemble in a lot of ways the exponential family. The literature on exponential dynamics is vast and there are many striking phenomena. It is expected that Zorich maps, given the higher dimensional setting and the greater flexibility,  should have an even more  intricate nature.  In this section we will mention some problems that require further study. 
		\subsection{Dynamics for different values of $\lambda$}
		We saw in Theorem \ref{main} that when $\lambda$ is large enough then the Julia set of the Zorich map is the entire $\mathbb{R}^3$ assuming that $\nu$ is large enough. It is interesting to ask what happens in the case when the scale factor $\lambda$ is not large. In that case we do not have enough expansion in the sense of Lemma \ref{dete} in order for our argument to work. Nonetheless, it seems that the dynamics in this case are also chaotic. So we ask \begin{question} Let $\lambda>0$. Does there always exist a constant $c_\lambda$ depending on $\lambda$ such that for all $\nu>c_\lambda$ the Julia set $\mathcal{J}(\mathcal{Z}_\nu)$ is the whole $\mathbb{R}^3$?
		\end{question}
		We can even ask this question in the complex plane. If we rescale the complex exponential  family we get the maps \[f_\nu(x+iy)=\nu\lambda e^x(\cos\left(\frac{y}{\lambda}\right)+i\sin\left(\frac{y}{\lambda}\right)),\] for $\lambda>0$ and $\nu\in\mathbb{R}$. Note that for $\lambda=1$ we get the exponential family. Of course those maps are no longer holomorphic but they are quasiregular and we can define their Julia set. A similar approach to that used for the Zorich maps should give us that the Julia set of those maps for $\lambda$ large enough  and $\nu>c'_\lambda$ is the entire complex plane,  where $c'_\lambda$ constant depending on $\lambda$. But we can ask \begin{question}Let $\lambda>0$.	Assuming that $\nu>c'_\lambda$ is $\mathcal{J}(f_\nu)$ the whole complex plane?
		\end{question}
		Closely related to the above question and worth mentioning here is the paper \cite{comduhr2019} where the authors study families of functions like $f_\nu$ in the complex plane. The functions they study are not necessarily quasiregular. However their results show that if we choose a $\lambda>0$ then for small values of $\nu$ the Julia set of $f_\nu$ is a 'Cantor bouquet'.   
		\subsection{Measurable dynamics of Zorich maps}
		In this subsection we assume that $\nu=1$ and $\lambda$ as in Theorem \ref{main}. We will make some remarks on the Lebesgue measure of some sets. First we need to introduce symbolic dynamics. In order to do that we partition $\mathbb{R}^3$ in the rectangular  beams \[T_{(i,j)}=T_{(0,0)}+2i(\lambda,\lambda,0)+2j(\lambda,-\lambda,0),\] where $i,j\in\mathbb{Z}$ and \begin{equation*}\begin{split}T_{(0,0)}=B_{(0,0)}\cup B_{(0,-1)}&\cup\{(x_1,x_2,x_3):x_1=x_2-2\lambda, -2\lambda< x_1\leq 0\}\\&\cup\{(x_1,x_2,x_3):x_1=-x_2+2\lambda, 0\leq x_1< 2\lambda\}\\&\cup\{(x_1,x_2,x_3):x_1=x_2, -\lambda< x_1< \lambda\}.\end{split}\end{equation*} For each point $x\in\mathbb{R}^3$ we associate a sequence on $\mathbb{Z}\times\mathbb{Z}$, $S(x):=(s_1,s_2,\dots)$ which we call its \textit{itinerary} and the $s_k=(s_{k,1},s_{k,2})$ are chosen such that $\mathcal{Z}^k(x)\in T_{s_k}$. We denote the space of all sequences by $\Sigma$ so $S$ is a map from $\mathbb{R}^3$ to $\Sigma$.  This procedure of course can be done to the exponential map in a similar manner. Consider now the set of all points with a given itinerary $(s_1,s_2,\dots)$. Namely the set \[\{x\in\mathbb{R}^3:S(x)=(s_1,s_2,\dots)\}.\] Ghys, Sullivan and Goldberg in \cite{Ghys1985} proved that the analogous set for the exponential map has Lebesgue measure zero. 
		With a bit more work we can see that in our proof of Theorem~\ref{main} we have actually proven the same result for the Zorich maps. Phrasing it in the same way as in \cite{Ghys1985} we have proven
		\begin{theorem}\label{fibers}
			The fibers of the map $S$ have Lebesgue measure zero.
		\end{theorem}  
		\begin{proof}
			Let $V$ be a set with $m(V)>0$ and all points in $V$ have the same itinerary $s$. Remember that the planes $x_1=\pm x_2$ together with their parallel translates form a forward invariant set so any point in $V$ which lands on one of those planes stays on those planes. Those points will have zero Lebesgue measure since the planes have zero Lebesgue measure and quasiregular maps have Luzin's N property (see \cite[I.Proposition 4.14]{Rickman}). Hence, we can assume that $V$ does not contain such points and it always stays on the interior of the square beams under iteration. 
			Thus we find ourselves in the same two cases as in the proof of Theorem \ref{main}. Note here that  Lemmas \ref{metro}, \ref{axis} require the set $V$ to be connected. However, it is easy to see in their proofs that this hypothesis can be weakened to all points in $V$ have the same itinerary which is exactly what we have here.
			
			The first case now of Theorem \ref{main} is exactly the same. Assuming that points in $V$ have an itinerary in which we can find a subsequence $s_{n_k}$ with $s_{n_k}\not=(0,0)$ we arrive at a contradiction due to the fact that $m(V)>0$.
			
			On the second case we assume that the itinerary of points in $V$ is eventually $0$ and without losing generality in fact equal to $(0,0,\dots)$. We may assume that all points in $V$ are density points since by Lebesgue's density theorem this is true for almost all points. Thus if  $x\in V$  then we know that \[\frac{m(B(x,\varepsilon)\cap V)}{m(B(x,\varepsilon))}>0,\] for all $\varepsilon>0$.
			
			We now claim that $m\left(B(x,\varepsilon)\cap V\right)>0$, for all $\varepsilon>0$ small enough if and only if $m(B(\mathcal{Z}(x),\varepsilon)\cap \mathcal{Z}(V))>0$ for all $\varepsilon>0$ small enough. Indeed, this follows by Lusin's N property and the fact that the Zorich map is locally invertible in $\mathcal{Z}(V)$.
			
			This implies that all points in $\mathcal{Z}^n(V)$ have the property $m(B(y,\varepsilon)\cap \mathcal{Z}^n(V))>0$ for all $\varepsilon>0$ small enough. Hence, by Lemma \ref{newlemma}, we may assume  that $x$ lies in $A_1$ (otherwise just consider an iterate of $V$ and rename that as $V$) and fix a small $\varepsilon$ so that $$U:=B(x, \varepsilon)\cap V \subset A_1.$$ We can now repeat the argument in the proof of the second case of Theorem \ref{main} and conclude that there is a subsequence $n_j$ with $\mathcal{Z}^{n_j}(U)\subset A_2\cup A_3$ but $m(\mathcal{Z}^{n_j}(U))\to\infty$ which is a contradiction due to the fact that $m(A_2\cup A_3)<\infty$.
		\end{proof}
		Closely related is the question of the typical behaviour of an orbit of the exponential map. Lyubich in \cite{Lyubich1988} proved that for Lebesgue almost all points of the complex plane the limit set of their orbit $E^n(z)$ is the orbit of $0$, $\{E^n(0)\}_{n\in\mathbb{N}}$ plus $\infty$. Thus a typical point will follow closely the orbit of $0$ for some time and then "break off" for some iterates until it goes back to following the orbit of $0$ for more iterates this time. Hence, almost all points in the complex plane belong to the bungee set of the exponential map (see \cite{Osborne2016}), namely the set of points that neither escape to infinity nor remain bounded under iteration. The bungee set can be also defined for quasiregular maps (see \cite{Nicks2018}). So we ask 
		
		\begin{question}
			What is the typical behaviour of an orbit of a point $x\in\mathbb{R}^3$ under the Zorich map? Do almost all points belong to the bungee set?
		\end{question}
		Another interesting question that was answered by Lyubich in the same paper is that of ergodicity of the exponential. Ergodicity here means that there is no partition of the plane in two invariant sets of positive Lebesgue measure. We have that 
		\begin{theorem}[(Lyubich \cite{Lyubich1988})]
			$E(z)$ is not ergodic.
		\end{theorem}
		In the same sense we can ask 
		\begin{question}
			Is the Zorich map $\mathcal{Z}$ ergodic?
		\end{question}
		\subsection{Indecomposable continua in Zorich maps}
		Another fascinating and well-studied phenomenon in exponential dynamics is the presence of indecomposable continua in the dynamic plane. It was Devaney in \cite{Devaney1993} who first studied such sets. The way to construct them in the complex plane is as follows. Consider the strip \[S=\{z\in\mathbb{C}:0\leq\imag z \leq \pi\}.\] Now take any $\kappa>1/e$ and consider the set \[\Lambda:=\{z\in\mathbb{C}:E_{\kappa}^n(z)\in S \hspace{1mm}\text{for all}\hspace{1mm}n\in\mathbb{N}\}.\] By suitably compactifing  this set then Devaney shows that we get a curve that accumulates everywhere on itself but does not separate the plane. Then by applying a theorem of Curry, \cite[Theorem 8]{Curry1991} he concludes that this curve must be an indecomposable continuum.\\
		
		Assuming that $\nu=1$, we can try and construct a similar set in the case of Zorich maps. The role of the strip $S$ is played now by the rectangular beam $\overline{B_{(0,0)}}$. Thus we can consider the set \[\Lambda_{\mathcal{Z}}:=\{x\in\mathbb{R}^3:\mathcal{Z}^n(x)\in \overline{B_{(0,0)}}, \hspace{1mm}\text{for all}\hspace{1mm}n\in\mathbb{N}\}.\] We can also, just like Devaney, suitably compactify this set and get a surface, let us call it $\Gamma$, that accumulates everywhere on itself. However the criterion of Curry is no longer available in this higher dimensional setting so Devaney's argument does not work here.
		\begin{question}
			Is $\Gamma$ an indecomposable continuum?
		\end{question}
		If the answer to the above question is yes we can then consider the same continua for different values of $\nu$. Let us call those continua $\Gamma_{\nu_1}$ and $\Gamma_{\nu_2}$, with $\nu_1$, $\nu_2\geq 1$.
		\begin{question}
			If $\nu_1\not=\nu_2$ are $\Gamma_{\nu_1}$ and $\Gamma_{\nu_2}$ homeomorphic?
		\end{question}

		Let us also remark here that the points in the set $\Lambda_{\mathcal{Z}}$ all have the same itinerary so by the results of the previous subsection we have that the three dimensional Lebesgue measure of this set is zero. 
		
		Finally, let us mention \cite{Devaney2002} where the authors prove the existence of many more indecomposable continua in the dynamical plane of the exponential map and ask many more questions. Such considerations also make sense for Zorich maps.
		
		\bibliographystyle{amsplain}
		\bibliography{bibliography}

\begin{thebibliography}{10}

\bibitem{Baker1968}
I.~N. Baker.
\newblock Repulsive fixpoints of entire functions.
\newblock {\em Math. Z.}, 104(3):252--256, 1968.

\bibitem{Berg1}
W.~Bergweiler.
\newblock Iteration of quasiregular mappings.
\newblock {\em Comput. Methods Funct. Theory}, 10:455--481, 2010.

\bibitem{bergk}
W.~Bergweiler.
\newblock Karpinska's paradox in dimension three.
\newblock {\em Duke Math. J.}, 154:599--630, 2010.

\bibitem{berg2013}
W.~Bergweiler.
\newblock Fatou-{Julia} theory for non-uniformly quasiregular maps.
\newblock {\em Ergodic Theory Dynam. Systems}, 33(1):1--23, 2013.

\bibitem{B-Nicks}
W.~Bergweiler and D.A. Nicks.
\newblock Foundations for an iteration theory of entire quasiregular maps.
\newblock {\em Israel J. Math.}, 201(1):147--184, 2014.

\bibitem{COMDUeHR2017}
Patrick Comd\"uhr.
\newblock On the differentiability of hairs for {Zorich} maps.
\newblock {\em Ergodic Theory Dynam. Systems}, 39(7):1824--1842, 2017.

\bibitem{comduhr2019}
Patrick Comd\"uhr, Vasiliki Evdoridou, and David~J. Sixsmith.
\newblock Dynamics of generalised exponential maps.
\newblock {\em preprint, arXiv:1904.11766}, 2019.

\bibitem{Curry1991}
Stephen~B. Curry.
\newblock One-dimensional nonseparating plane continua with disjoint
  $\epsilon$-dense subcontinua.
\newblock {\em Topology Appl.}, 39(2):145--151, 1991.

\bibitem{Devaney1993}
Robert~L. Devaney.
\newblock Knaster-like continua and complex dynamics.
\newblock {\em Ergodic Theory Dynam. Systems}, 13(4):627--634, 1993.

\bibitem{Devaney2003}
Robert~L. Devaney.
\newblock {\em An Introduction To Chaotic Dynamical Systems}.
\newblock Taylor and Francis Inc, 2003.

\bibitem{Devaney2010}
Robert~L. Devaney.
\newblock Complex exponential dynamics.
\newblock In {\em Handbook of Dynamical Systems}, pages 125--223. Elsevier,
  2010.

\bibitem{Devaney2002}
Robert~L. Devaney and Xavier Jarque.
\newblock Indecomposable continua in exponential dynamics.
\newblock {\em Conform. Geom. Dyn.}, 6:1--12, 2002.

\bibitem{Devaney1984}
Robert~L. Devaney and Michal Krych.
\newblock Dynamics of $exp(z)$.
\newblock {\em Ergodic Theory Dynam. Systems}, 4(1):35--52, 1984.

\bibitem{Eremenko1992}
A.~Eremenko and M.~Yu Lyubich.
\newblock Dynamical properties of some classes of entire functions.
\newblock {\em Ann. Inst. Fourier}, 42(4):989--1020, 1992.

\bibitem{Eremenko}
A.~E. Eremenko.
\newblock On the iteration of entire functions.
\newblock {\em Banach Center Publ.}, 23:339--345, 1989.

\bibitem{Fletcher2013}
Alastair~N. Fletcher and Daniel~A. Nicks.
\newblock Chaotic dynamics of a quasiregular sine mapping.
\newblock {\em J. Difference Equ. Appl.}, 19(8):1353--1360, 2013.

\bibitem{Ghys1985}
Etienne Ghys, Lisa~R. Goldberg, and Dennis~P. Sullivan.
\newblock On the measurable dynamics of $z \rightarrow e^z$.
\newblock {\em Ergodic Theory Dynam. Systems}, 5(3):329--335, 1985.

\bibitem{Iwaniec2001}
T.~Iwaniec and G.~Martin.
\newblock {\em Geometric Function Theory and Non-linear Analysis}.
\newblock Oxford mathematical monographs. Oxford University Press, 2001.

\bibitem{Karpinska}
B.~Karpinska.
\newblock Area and hausdorff dimension of the set of accessible points of the
  {Julia} sets of $\lambda e^z$ and $\lambda \sin z$.
\newblock {\em Fund. Math.}, 159:269--287, 1999.

\bibitem{Karpinska1999}
B.~Karpi{\'{n}}ska.
\newblock Hausdorff dimension of the hairs without endpoints for $\lambda$ exp
  z.
\newblock {\em Comptes Rendus de l'Acad{\'{e}}mie des Sciences - Series I -
  Mathematics}, 328(11):1039--1044, 1999.

\bibitem{Lyubich1988}
M.~Yu. Lyubich.
\newblock Measurable dynamics of the exponential.
\newblock {\em Sib. Math. J.}, 28(5):780--793, 1988.

\bibitem{milnor}
J.~Milnor.
\newblock {\em Dynamics in One Complex Variable}.
\newblock Annals of Mathematics Studies. Princeton University Press, 2006.

\bibitem{Misiu}
M.~Misiurewicz.
\newblock On iterates of $e^z$.
\newblock {\em Ergodic Theory Dynam. Systems}, 1:103--106, 1981.

\bibitem{N-S}
D.~A. Nicks and D.~J. Sixsmith.
\newblock Periodic domains of quasiregular maps.
\newblock {\em Ergodic Theory Dynam. Systems}, 38:2321--2344, 2018.

\bibitem{Nicks2018}
Daniel~A. Nicks and David~J. Sixsmith.
\newblock The bungee set in quasiregular dynamics.
\newblock {\em Bull. Lond. Math. Soc.}, 51(1):120--128, 2018.

\bibitem{Osborne2016}
John~W. Osborne and David~J. Sixsmith.
\newblock On the set where the iterates of an entire function are neither
  escaping nor bounded.
\newblock {\em Ann. Acad. Sci. Fenn. Math.}, 41:561--578, 2016.

\bibitem{REMPE2010}
Lasse Rempe.
\newblock The escaping set of the exponential.
\newblock {\em Ergodic Theory Dynam. Systems}, 30(2):595--599, 2010.

\bibitem{Rickman}
S.~Rickman.
\newblock {\em Quasiregular mappings}, volume~26 of {\em Ergeb. Math.
  Grenzgeb.}
\newblock Springer-Verlag, Berlin, 1993.

\bibitem{Shen2015}
Zhaiming Shen and Lasse Rempe-Gillen.
\newblock The exponential map is chaotic: An invitation to transcendental
  dynamics.
\newblock {\em Amer. Math. Monthly}, 122(10):919--940, 2015.

\bibitem{tsantaris2021}
Athanasios Tsantaris.
\newblock {Explosion points and topology of Julia sets of Zorich maps}.
\newblock {\em preprint, arXiv:2107.13322}, 2021.

\bibitem{Silva1988}
M.~{Viana da Silva}.
\newblock The differentiability of the hairs of exp(z).
\newblock {\em Proc. Amer. Math. Soc.}, 103(4):1179, 1988.

\bibitem{vuorinen}
M.~Vuorinen.
\newblock {\em Conformal geometry and quasiregular mappings}, volume 1319 of
  {\em Lecture Notes in Math.}
\newblock Springer-Verlag, Berlin, 1988.

\bibitem{Zorich}
V.~A. Zorich.
\newblock The theorem of {M. A. Lavrent'ev} on quasiconformal mappings in
  space.
\newblock {\em Mat. Sb.}, 74:417--433, 1967.

\bibitem{Zorich2019}
V.~A. Zorich.
\newblock {\em Mathematical Analysis II}.
\newblock Springer Berlin Heidelberg, 2019.

\end{thebibliography}
	\end{document}